\documentclass[a4paper]{amsart}

\newtheorem{theorem}{Theorem}[section]
\newtheorem{lemma}[theorem]{Lemma}
\newtheorem{proposition}[theorem]{Proposition}
\newtheorem{corollary}[theorem]{Corollary}

\theoremstyle{definition}

\theoremstyle{remark}
\newtheorem{remark}[theorem]{Remark}

\numberwithin{equation}{section}

\usepackage[T1]{fontenc}
\usepackage[utf8]{inputenc}
\usepackage{graphicx}
\usepackage{amssymb}
\usepackage{amsfonts}
\usepackage{amsmath}
\usepackage{hyperref}
\usepackage{enumitem}
\usepackage{etoolbox}
\usepackage{mathtools}
\usepackage[dvipsnames]{xcolor}
\usepackage{aligned-overset}
\usepackage{bm}
\usepackage{wasysym} %for conc
\usepackage[most]{tcolorbox}
\usepackage{tikz}
\usepackage{dsfont}
\newcommand{\calG}{{\mathcal G}}
\newcommand{\calK}{{\mathcal K}}
\newcommand{\calW}{{\mathcal W}}

\usetikzlibrary{matrix}
\let\vec\oldvec

\let\epsilon\varepsilon

\let\phi\varphi

\DeclareMathOperator{\spn}{span}

\DeclareMathOperator{\dist}{dist}

\DeclareMathOperator{\vec}{\mathbf{vec}}

\DeclareMathOperator{\depth}{depth}
\DeclareMathOperator{\size}{size}

\DeclareMathOperator{\Realiz}{{\mathrm{R}}}
\DeclareMathOperator{\bRealiz}{\boldsymbol{{\mathrm{R}}}}

\newtheorem{assumption}{Assumption}
\newtheorem{setting}{Setting}
\newcommand{\bvec}[1]{\bm{#1}}
\newcommand{\bmat}[1]{\mathbf{#1}}

\newcommand{\tbc}{\bvec{\tilde{c}}}
\newcommand{\bck}{\bc^{v,(k)}}

\newcommand{\tbck}{\tbc^{v,(k)}}

\newcommand{\bcK}{\bc^{v,(K)}}

\newcommand{\bckpone}{\bc^{v,(k+1)}}

\newcommand{\tbckpone}{\tbc^{v,(k+1)}}

\newcommand{\analysis}{\mathsf{A}}

\newcommand{\uvNkpone}{u^{v, (k+1)}_N}

\newcommand{\uvNk}{u^{v, (k)}_N}
\newcommand{\cS}{\mathcal{S}}
\newcommand{\tcS}{\widetilde{\mathcal{S}}}

\newcommand{\hpsi}{\hat{\psi}}
\newcommand{\lnorm}{\left\vert\kern-0.25ex\left\vert\kern-0.25ex\left\vert}
\newcommand{\rnorm}{\right\vert\kern-0.25ex\right\vert\kern-0.25ex\right\vert}

\newcommand{\tC}{{\widetilde{C}}}
\newcommand{\tZ}{{\widetilde{Z}}}

\newcommand{\tL}{{\widetilde{L}}}

\newcommand{\cK}{\mathcal{K}}

\newcommand{\cL}{\mathcal{L}}
\newcommand{\cG}{\mathcal{G}}

\newcommand{\tcG}{\widetilde{\mathcal{G}}}
\newcommand{\cE}{\mathcal{E}}

\newcommand{\cI}{\mathcal{I}}
\newcommand{\cT}{\mathcal{T}}

\newcommand{\bx}{\bvec{x}}

\newcommand{\by}{\bvec{y}}

\newcommand{\bzero}{{\bvec{0}}}

\newcommand{\Clip}{C_{\mathrm{Lip}}}
\newcommand{\tClip}{\widetilde{C}_{\mathrm{Lip}}}

\newcommand{\cA}{\mathcal{A}}

\newcommand{\cR}{\mathcal{R}}

\newcommand{\cD}{\mathcal{D}}

\newcommand{\tcK}{\widetilde{\mathcal{K}}}

\newcommand{\cU}{\mathcal{U}}
\newcommand{\cO}{\mathcal{O}}

\newcommand{\N}{\mathbb{N}}

\newcommand{\Id}{\mathbf{Id}}

\newcommand{\bA}{\bmat{A}}
\newcommand{\bB}{\bmat{B}}
\newcommand{\hbB}{\bmat{\widehat{B}}}
\newcommand{\bC}{\bmat{C}}
\newcommand{\bD}{\bmat{D}}

\newcommand{\bX}{\bmat{X}}
\newcommand{\bH}{\bmat{H}}

\newcommand{\bc}{\bvec{c}}

\newcommand{\hbff}{\bvec{\hat{f}}}

\newcommand{\Phiappx}{\Phi^{\mathrm{app}}}
\newcommand{\Phirec}{\Phi^{\mathrm{dec}}}
\newcommand{\Phistep}{\Phi^{\mathrm{step}}}
\newcommand{\Phiit}{\Phi^{\mathrm{it}}}
\newcommand{\Phiinput}{\Phi^{\hbB}_{N, M}}

\newcommand{\epsit}{\epsilon_{\mathrm{it}}}
\newcommand{\epsapp}{\epsilon_{\mathrm{app}}}
\newcommand{\epsrec}{\epsilon_{\mathrm{dec}}}

\newcommand{\frab}{\mathfrak{b}}

\newcommand{\sconc}{\odot}

\newcommand{\R}{\mathbb{R}}

\newcommand{\C}{\mathbb{C}}

\newcommand{\hPsi}{\widehat{\Psi}}

\newcommand{\cDab}{\cD_{\alpha, \beta}}
\newcommand{\cKab}{\cK_{\alpha, \beta}}
\newcommand{\cKabm}{\cK_{\alpha, \beta}^m}

\newcommand{\tbeta}{\tilde{\beta}}

\newcommand{\NN}{\mathsf{NN}}
\newcommand{\bn}{\bm{n}}

\newcommand{\brho}{{\bm{\varrho}}}

\DeclareMathOperator{\ReLU}{ReLU}

\author{Carlo Marcati}
\address{Dipartimento di Matematica ``Felice Casorati'',
  University of Pavia, Italy}
\email{carlo.marcati@unipv.it}
\thanks{CM is a member of GNCS -- INdAM and acknowledges the support of the Italian Ministry of University and Research
  (MUR) through the PRIN 2022 PNRR project NOTES (No. P2022NC97R), funded by the
  European Union -- Next Generation EU and the PRIN 2022 project ASTICE
  (No. 202292JW3F)}
\author{Christoph Schwab}
\address{Seminar for Applied Mathematics, 
  ETH Z\"urich, Switzerland}
\email{christoph.schwab@sam.math.ethz.ch}

\title[Expression rates of Neural Operators for Elliptic PDEs in Polytopes]{
Expression Rates of Neural Operators 
\\
for Linear Elliptic PDEs in Polytopes}


\begin{document}
\maketitle
% \tableofcontents
\begin{abstract}%
We study the approximation rates of a class of 
deep neural network approximations 
of operators which arise as 
data-to-solution maps $\cS$ 
of linear elliptic partial differential equations (PDEs),
and act between pairs $X,Y$ of suitable infinite-dimensional spaces.
We prove expression rate bounds for
approximate neural operators $\cG$ with the structure
$\cG = \cR \circ \cA \circ \cE$, with linear encoders $\cE$ and decoders $\cR$.
We focus in particular on deepONets emulating the coefficient-to-solution maps for 
elliptic PDEs set in polygons and in some polyhedra.
Exploiting the regularity of the solution sets of elliptic PDEs in polytopes, 
we show algebraic rates of convergence for problems with data with finite
regularity, and exponential rates for analytic data.
\end{abstract}
%
%%%%%%%%%%%%%%%%%%%%%%%%%%%%%%%%%%%%%%%%%%%%%%%%%%%%%%%%%%%%%%%%%%%%%%%%%%%%%%%%%%%%%%
\section{Introduction}
\label{sec:Intro}
%%%%%%%%%%%%%%%%%%%%%%%%%%%%%%%%%%%%%%%%%%%%%%%%%%%%%%%%%%%%%%%%%%%%%%%%%%%%%%%%%%%%%%
Neural Networks (NNs) have been shown to exhibit 
approximation properties which are at least on par with all ``traditional''
approximation architectures. 
Accordingly, deep NNs
have recently been leveraged for the numerical approximation of PDE 
solutions.  
We mention only 
``Physics-informed NNs'' and their variants,
``Deep Ritz Methods'' and ``Deep Least Squares'',
see \cite{DM24_1088} and the references there 
for a taxonomy of NN based PDE approximations.
In these methodologies, 
for a given set of data
the PDE solutions are numerically approximated
by NNs, with NN training corresponding to
numerical minimization of suitable loss functionals based on
suitable residuals of the PDE of interest.

More recent approaches aim at learning
NN \emph{surrogates of data-to-solution maps $\cS$ for PDEs}.
These methodologies are broadly referred to as
\emph{neural operators} (NOs).
The state of the art of the use of 
neural operators in SciML at the time of this writing
is surveyed in \cite{kovachki2024operator,DM24_1088}
and in the references there.

The present paper is devoted to the
expression rate analysis
of a class of neural operators for the emulation of the (nonlinear) 
coefficient-to-solution maps $\cS:\cK\to Y$ for boundary value problems for
linear, self-adjoint second order elliptic differential operators.
Here,  $\cK$ is a compact subset of the space
of admissible PDE input data,
and $Y$ is a Hilbert space where suitable weak formulations
of the PDE admit a unique solution.
Specifically, we prove existence of approximating, finite-parametric
neural operators $\{\calG_\epsilon\}_{\epsilon > 0}$ of accuracy $\epsilon\in(0,1)$, 
with an encoder-approximator-decoder structure (see \eqref{eq:G=RAE} ahead).
We bound the size and depth of NNs forming the operator with respect to the
worst-case error $\epsilon \in (0,1)$ over the data set $\cK$.
When the domain where the PDE is set is polytopal, 
we show algebraic and exponential expression rates of $\calG_\epsilon$ 
for PDEs with input data with, respectively, finite and analytic regularity.
The intent of this paper is to derive upper bounds for the approximation rates.
A construction is introduced, but it is used as a theoretical tool, rather
than being proposed as a practical tool to be used in computations. In
practice, neural networks are trained with, e.g., gradient descent based on an
empirical loss function computed on finite training data, introducing 
training and generalization errors which we do not consider here.
\subsection{Existing Results}
\label{sec:ExRes}
For the approximation of $\cS$ 
by NOs a first, basic question is \emph{universal approximation}: will NOs,
with increasing size, be able to capture the data-to-solution map $\cS$ 
in norms relevant for the physical problem
to any prescribed accuracy $\epsilon>0$?
For many NOs on many different types of function spaces, universal approximation
has been proved at this point. 
We refer to
\cite{ChenChen1993,schwab2023deep,kratsios2023approximation,furuya2023globally,lanthaler2023nonlocal}
and the references there.

A different line of works addresses, in, to some extent, rather specific settings,
quantitative bounds on the approximation rate of NOs.
As a rule, to derive these approximation rate bounds,
stronger assumptions are imposed on some or all of the following items:
(i) input and output regularity, 
(ii) mapping properties of $\cS$,
and 
(iii) architectures.
We refer to \cite{deng2021convergence,SchwabZech,kovachki2021neural,Kovachki2021,lanthaler2021,Kutyniok2022,herrmann2024neural,Lanthaler2023,schwab2023deep,MS2023}, 
and to the survey \cite{kovachki2024operator} 
and the references there.
In \cite{Lanthaler2023}, 
lower bounds for neural operator approximations with PCA-based en- and decoder
have been shown to require input and output regularity,
quantified in terms of decay rates of principal components.

More closely related to the techniques we use here is \cite{Kutyniok2022}.
There, 
the authors consider parametric PDEs, and in particular the
mapping from finitely many parameters to the coefficients of the
solution on a (reduced or finite element) basis that is defined a priori. 
Here, we extend that
analysis and consider infinite dimensional input spaces and 
NN-based basis functions.
It follows that we have to deal with two additional sources of error which
derive from the truncation of the encoded input and from the approximated decoding.

We alert the reader that the architecture \eqref{eq:G=RAE} considered below 
assumes linear decoder $\cR$.
As a result, the approximation properties of the neural operators
are constrained by the \emph{Kolmogorov barrier}, i.e., 
high convergence rates
require that the solution sets can be well approximated by linear spaces.
An emerging field in SciML and model order reduction (MOR)
(see, e.g.,
\cite{Cohen2022,Dalery2023,Ehrlacher2020,Franco_2022,Lee2020,Regazzoni2024,Rim2023,Seidman2022})
has its focus on operator
surrogates with nonlinear decoders, suited for problems with slowly decaying Kolmogorov $N$-widths.  
See the definition in \eqref{eq:nwidth} below. 
%%%%%%%%%%%%%%%%%%%%%%%%%%%%%%%%%%%%%%%
\subsection{Contributions}
\label{sec:Contr}
%%%%%%%%%%%%%%%%%%%%%%%%%%%%%%%%%%%%%%%
The main contribution of this paper is in Section~\ref{sec:polytopes}, where
we consider the approximation of the
coefficient-to-solution map $\cS: X\to Y$
of a linear, second order elliptic,
divergence-form PDE in a bounded polytopal domain.
We establish %, for a model, linear elliptic PDE, 
the existence of neural operators $\cG_\epsilon$
which approximate $\cS$ on compact subsets of data
to any target accuracy $\epsilon\in (0, 1)$ in the norm of the space
$C(X, Y)$ of continuous data-to-solution maps for the PDE.
The neural operators analyzed have an encoder-approximator-decoder
architecture, 
with the approximation and decoding maps involving fully connected, 
feedforward NNs, see \eqref{eq:deepONet-intro}.
% This is the deepONet
% architecture introduced in \cite{lu2020deeponet}.
In terms of the number of neurons in $\cG_\epsilon$,
we show algebraic emulation rates when the
data has finite regularity (see Section \ref{sec:algconv}) 
and exponential emulation rates
when the data is analytic (see Section \ref{sec:ExpCnvAnDat}). 
In both cases, we exploit the structure of the solutions to
elliptic PDEs in polytopes and convergence rates for their approximation by
piecewise polynomials.
\subsection{Notation}
\label{sec:Notn}
We collect symbols and notation to be used throughout the rest of this paper.
%%%%%%%%%%%%%%%%%%%%%%%%%%%%%%%%%%%%%%%%%%%%%%%%%%%%%%%%%%%%%%%%%%%%%%%%%%55
\subsubsection{General Notation}
\label{sec:NotGen}
We denote $\N =\{1, 2, 3, \dots\}$ and write $\N_0 = \{0\}\cup \N$. 
As usual, $\R$ and $\C$ shall denote real and complex numbers.

For an integer $k\geq 2$, 
we denote the set $\{ n\in \N: n\geq k \}$ as $\N_{\geq k}$.
With a basis $\Phi_M = \{\phi_1, \dots, \phi_M\}$ 
for the $M$-dimensional space $\spn\{\phi_1, \dots, \phi_M\}$,
we denote its 
analysis operator 
$\analysis^{\Phi_M} : \spn(\Phi_M) \to \R^M$, i.e.
\begin{equation*}
  \analysis^{\Phi_M} : v \mapsto \{c_1, \dots, c_M\}, 
  \qquad \text{if}\qquad v = \sum_{i=1}^M c_i \phi_i.
\end{equation*}
Rectangular matrices and, more generally, 
$k$-arrays of tensors of order $k\geq 2$
with real-valued entries shall be denoted by boldface letters: 
$\bD \in \R^{n_1\times n_2\times ... \times n_k}$.
Given $\bc \in \R^M$ and $\bC\in\R^{N\times M}$, 
we denote $\bc\cdot\Phi_M = \sum_{i=1}^Mc_i \phi_i$ 
and write
\begin{equation*}
 \bC\Phi_M  = \left\{ \sum_{i=1}^M\bC_{1i}\phi_i, \dots, \sum_{i=1}^M\bC_{Ni}\phi_i\right\}.
\end{equation*}

For normed spaces $X$ and $Y$ 
and continuous $\mathcal{T} : X \to Y$ bounded on $\cK\subset X$,
we write
\begin{equation*}
\| \mathcal{T}  \|_{L^\infty(\cK; Y)} = \sup_{f\in \cK} \| \mathcal{T}(f) \|_Y.
\end{equation*}

The Euclidean scalar product of two vectors $x,y\in \R^d$ 
shall be denoted with $x\cdot y = x^\top y$. 
The corresponding Euclidean vector norm on $\R^d$ is 
denoted with $\| \cdot \|_{\ell_2}$, i.e. $\|x\|^2_{\ell_2} = x\cdot x$.
For a matrix $\bA \in \R^{N\times M}$, we denote by $\vec(\bA) \in \R^{NM}$
its reshaping into a vector with entries
\begin{equation*}
  \vec(\bA)_{i + N(j-1)} = \bA_{i, j}, \qquad (i, j) \in \{1, \dots, N\} \times \{1, \dots, M\}.
\end{equation*}
We also write
\begin{equation*}
  \|\bA \|_0 = \left| \{( i, j ): \bA_{ij}\neq 0\}\right|,
\end{equation*}
with $| \cdot|$ denoting the cardinality of a set.
Given $\Omega\in \R^d$, a continuous function $a:\Omega\to\R$, and 
points $\bX = \{\bx_1, \dots, \bx_M\}$ in $\Omega$, we write $a(\bX)=
\{a(\bx_1), \dots, a(\bx_M)\}$.
For Banach spaces $Y,Z$, we denote by $\cL(Y,Z)$ 
the set of bounded, linear maps $L:Y\to Z$.
%%%%%%%%%%%%%%%%%%%%%%%%%%%%%%%%%%%%%%%%%%%%%%%%%%%%%%%%%%%%%%%%
\subsubsection{Function spaces}
\label{sec:Spces}
Let $\Omega\subset \R^d$ be a bounded domain,
with Lipschitz boundary $\partial \Omega$.
For $m\in \N_0$ and $p\in [1,\infty]$,
we use the standard notation $W^{m, p}(\Omega)$ 
for the Sobolev space of order $m$ and Lebesgue summability $p$ of
functions defined on $\Omega$, with the 
shorthand $H^m(\Omega)=W^{m.2}(\Omega)$. 
For a Banach space $X$, we denote by $X'$ its topological dual.

All Banach spaces under consideration will be over the reals; for
a Banach space $X$ over $\R$, we define $X^{\mathbb{C}} = X \otimes \{ 1, i\}$
its ``complexification'', which we assume 
equipped with a norm $\| \cdot \|_{X^{\mathbb{C}}}$
extending $\| \cdot \|_{X}$ (see, e.g., \cite{MunozCplxB}).
%
%%%%%%%%%%%%%%%%%%%%%%%%%%%%%%%%%%%%%%%%%%%%%%%%%%%%%%%%%%%%%%%%%%%%%
\subsubsection{Neural Networks}
\label{sec:NNs}
For integers
$L\in \N$ 
and 
$\{n_0, \dots, n_{L}\} \in \N^{L+1}$, 
we define a deep neural network (NN)
of depth $L$ and 
widths $\bn = \{n_0, \dots, n_{L}\}$
as a finite list of weight matrices and bias vectors, 
i.e.,
\begin{equation}
  \label{eq:DefNN}
  \Phi = (A_\ell, b_\ell)_{\ell=1}^L \in
  \bigtimes_{\ell=1}^L (\R^{n_{\ell}\times n_{\ell-1}} \times \R^{n_\ell}) \eqqcolon \NN_{L, \bn}.
\end{equation}
We write, 
for $\Phi$ in $\NN_{L,\bn}$,
\begin{equation*}
\depth(\Phi) = L, \qquad 
\size(\Phi) = \sum_{\ell=1}^L\left(\| A_\ell\|_0 + \|b_\ell\|_0 \right).
\end{equation*}
When $L\geq 2$,
for given activation functions
$\rho_i :\R\to\R $, $i=1,\dots, L-1$
(with the convention of acting on vector valued inputs component-wise),
we associate to the NN $\Phi$ 
and the list $\brho = (\rho_i)_{i=1}^{L-1}$ of activation functions,
the \emph{realization of the NN $\Phi\in \NN_{L, \bn}$} 
\begin{equation*}
  \Realiz_{\brho}(\Phi) =
  T_L \circ \rho_{L-1} \circ T_{L-1} \circ\dots \circ \rho_1 \circ T_1
\end{equation*}
where 
\begin{equation*}
T_\ell : x\mapsto A_\ell x+b_\ell\text{ for }\ell\in\{1, \dots, L\}.
\end{equation*}
Vector-valued realization maps
shall be denoted by a boldface symbol, e.g., $\bRealiz(\Phi)$.

If all activation functions coincide and equal, e.g., some function $\rho$,
i.e., if
$\rho_i = \rho$ for all $i=1, \dots, L-1$, we write $\Realiz_\rho$ for $\Realiz_{\brho}$.
When $\Phi = (A, b)\in \NN_{1, \bn}$ for any $\bn\in\N^2$, 
we simply write
$\Realiz(\Phi): x\mapsto Ax+b$.
The activation functions used in this paper will be 
any one of the following 
\begin{equation*}
% \ReLU : x\mapsto \max(x, 0) ,\qquad 
\ReLU^r:x\mapsto \max(x, 0)^r,
\end{equation*}
where $r \in \{1,2\}$ and we identify $\ReLU = \ReLU^1$.
Whenever we make generic statements valid for all these activations, 
or when the choice of activation $\rho$ is clear from the context,
we write $\Realiz(\Phi)$. 
All NNs in this work are feedforward NNs.
%%%%%%%%%%%%%%%%%%%%%%%%%%%%%%%%%%%%%%%%%%%%%
\subsubsection{Neural Operators}
\label{sec:neural-op-intro}
Neural operators are finite-parametric, 
computational approximations of 
continuous maps $\cS: X\to Y$ 
between Banach spaces $X$ and $Y$.

Specifically,
we consider the approximation of $\cS$ 
by
neural operators $\cG: X\to Y$
    which can be written as 
composition of three operators:
encoder $\cE$, approximator $\cA$, and decoder $\cR$, i.e., 
\begin{equation}
\label{eq:G=RAE}
  \cG = \cR \circ \cA \circ \cE,
\end{equation}
where, for suitable (eventually accuracy-dependent) dimensions $n, m\in \N$, 
\begin{equation*}
  \cE : X \to \R^m, \quad \cA: \R^m\to\R^n, \quad \cR: \R^n\to Y.
\end{equation*}
The structure \eqref{eq:G=RAE} was recently considered, e.g., 
in 
\cite{lanthaler2021,Lanthaler2023,herrmann2024neural,lanthaler2023nonlocal,lanthaler2024parametric}.

In the architecture \eqref{eq:G=RAE}, 
$\cA\circ\cE$ and $\cR$ are correspond, respectively, to the
``trunk'' net and ``branch'' net of the so-called ``deepONet'' 
architecture introduced in \cite{lu2020deeponet}. 
DeepONets are neural operators with the structure \eqref{eq:G=RAE}
where the (linear) decoder is given by
\begin{equation}
\label{eq:Fdec}
  \cR : \bx \mapsto \bx\cdot \bRealiz(\Phirec),
\end{equation}
where $\Phirec$ is a feedforward neural network, the encoder corresponds to
point evaluations on a set of points $\bX$, and the approximator $\cA$ is the
realization of a NN $\Phiappx$. 
This gives neural operators $\cG$ defined on subsets of continuous functions such that
\begin{equation}
  \label{eq:deepONet-intro}
  \cG : a\mapsto
  \bRealiz\left( \Phiappx \right)(a(\bX))\cdot\bRealiz(\Phirec).
\end{equation}

\subsubsection{Distance between sets and spaces and Kolmogorov $N$-widths}
\label{sec:Nwidths}
For a normed, linear space $X$ with norm $\| \cdot \|_X$ 
and for a subset $K\subset X$, 
for $N=1,2,...$ 
Kolmogorov $N$-widths are defined as 
\begin{equation}
  \label{eq:nwidth}
  d_N( K , X) = \inf_{\dim(V) = N} \dist_X(K, V) 
\;.
\end{equation}
Here,
\begin{equation*}
  \dist_X(K, V) = \sup_{x\in K}\dist_X(x, V) = \adjustlimits\sup_{x\in K}\inf_{v\in V} \| x-v\|_X
\end{equation*}
and the infimum in \eqref{eq:nwidth} is taken over
all $N$-dimensional linear subspaces $V$ of $X$. 
We recall here a result on transformation of 
Kolmogorov $N$-widths with algebraic decay under holomorphic mappings. 
This will be used to show Theorem \ref{th:ON-input-Nwidth}, 
under the additional assumption that the data-to-solution operator $\cS$ 
admits a holomorphic extension to the complex domain.
%%%%%%%%%%%%%%%%%%%%%%%%%%%%%%%%%%%%%%%%%%%%%%%%%%
\begin{lemma}[{\cite[Theorem 1.1]{Cohen2016}}]
  \label{lemma:Nwidth-holomorphic-general}
 Let $X,Y$ be complex Banach spaces, $O\subset X$ open, $K\subset O$ compact,
 and let $\cG$ be a holomorphic mapping from $O$ to $Y$ which is uniformly
 bounded on $O$. 
 If there exists $s>0$ such that 
 \begin{equation*}
   \sup_{N\in\N} N^s d_N(K, X) < \infty,
 \end{equation*}
 it holds for any $t<s-1$
 \begin{equation*}
   \sup_{N\in\N} N^td_N(\cG(K), Y) < \infty.
 \end{equation*}
\end{lemma}
%%%%%%%%%%%%%%%%%%%%%%%%%%%%%%%%%%%%%%%%%%%%%%%%%%%%%%%%%%%%%%%%%%%%%%%

\subsection{Layout} 
\label{sec:outline}
In Section~\ref{sec:PrbForm}, we introduce an abstract operator equation
with parametric operators, modeled on the simple linear, second order 
divergence-form PDE. 
In particular, we deal with 
boundary value problems for variational, elliptic self-adjoint differential 
operators with nonconstant coefficients. 

Section~\ref{sec:it} recalls Richardson iterations, for
the computation of finite-parametric approximations of the solutions.
These iterations are used in Section~\ref{sec:GenBounds} 
to derive error bounds for the approximation part of the
neural operator with architecture given by \ref{eq:G=RAE}.
The results are abstract at this stage, applying
to general, variational formulations of elliptic PDEs.

Section~\ref{sec:polytopes} addresses the specific setting of linear, 
second order elliptic PDEs in polygons and in some polyhedra.
We consider neural operators with the deepONet architecture \eqref{eq:deepONet-intro}
We consider two separate cases:
data with finite regularity and analytic data. 
We show, respectively, algebraic and exponential approximation rates.
%%%%%%%%%%%%%%%%%%%%%%%%%%%%%%%%%%%%%%%%%%%%%%%%%%%%%%%%%%%%%%%%%%%%%%%%%%%%%%5
\section{Forward Problem Formulation}
\label{sec:PrbForm}
%%%%%%%%%%%%%%%%%%%%%%%%%%%%%%%%%%%%%%%%%%%%%%%%%%%%%%%%%%%%%%%%%%%%%%%%%%%%%%
We introduce a variational formulation of the parametric forward problems
whose data-to-solution maps are
to be subsequently emulated by the operators
$\calG$ as in \eqref{eq:G=RAE}. 
We illustrate the scope of problems with several examples.
%%%%%%%%%%%%%%%%%%%%%%%%%%%%%%%%%%%%%%%%%%%%%%%%%%%%%%%%%%%%%%%%%%%%%%%%%%%%%%5
\subsection{Variational Formulation. Existence and Uniqueness}
\label{sec:VarFrm}
%%%%%%%%%%%%%%%%%%%%%%%%%%%%%%%%%%%%%%%%%%%%%%%%%%%%%%%%%%%%%%%%%%%%%%%%%%%%%%5
For the data space $X$, a real Banach space, 
and a solution space $Y$ which is a real Hilbert space, 
consider
for each $a\in X$, a self-adjoint operator $L(a)\in \cL(Y,Y')$. 
For all $a\in X$ and all $u, v\in Y$, 
with $\langle \cdot,\cdot \rangle$
denoting the $Y'\times Y$ duality pairing,
we introduce the bilinear form
\begin{equation}\label{eq:Defb}
  \frab(a; u, v) \coloneqq \langle L(a)u, v \rangle.
\end{equation}
Self-adjointness of $L(a)$ implies symmetry of $\frab(a;.,.)$.
We fix a nominal input $a_0\in X$ (with, for ease of notation, $\|a_0\|_X=1$) 
such that $L(a_0)$ is positive definite 
(sufficient conditions
will be provided with the set $\cDab\subset X$ in \eqref{eq:cDab} below).
We accordingly may 
associate with this $a_0$ the ``energy norm'' $\| \cdot \|_Y$
given by
\begin{equation}\label{eq:a0}
\| u \|^2_Y \coloneqq \frab(a_0; u, u), \;\; u\in Y\;.
\end{equation}
We suppose that 
the map $a\mapsto L$ is linear from $X$ to $\mathcal{L}(Y, Y')$ 
and
that
\begin{equation}
  \label{eq:bilinear-cont}
  | \frab(a; u, v) | \leq \|a\|_X\|u\|_Y\|v\|_Y \qquad \forall a\in X,\, u, v\in Y.
\end{equation}
% \end{setting}
% }
For constants $\alpha, \beta\in \R$ such that $0<\beta<\alpha<\infty$, 
let the set $\cDab \subset X$ 
of admissible data be given by
\begin{multline}
  \label{eq:cDab}
  \cDab \coloneqq \big\{ a \in X: (\alpha-\beta) \|w\|^2_Y 
  \leq 
  \frab(a; w, w), \\
  |\frab(a; v, w)|  \leq (\alpha+\beta) \|v \|_Y \|w\|_Y, \, \forall v,w\in Y \big\}.
\end{multline} 
Fix $0\ne f\in Y'$; 
for each $a\in \cDab$, 
define $u^a \in Y$ to be the weak solution to
\begin{equation}
  \label{eq:problem}
  \frab(a; u^a, v) =  \langle f, v \rangle\qquad \forall v\in Y.
\end{equation}
Note that, by definition, $\alpha a_0 \in \cDab$.
\begin{remark}
  \label{remark:example}
  As an example of this setting, consider, for $\Omega\subset\R^d$ the spaces
  $X=L^\infty(\Omega)$ and $Y=H^1_0(\Omega)$ and the elliptic PDE
  \begin{equation}
    \label{eq:example-PDE}
    -\nabla \cdot(a\nabla u) =f
  \end{equation}
  with homogeneous Dirichlet boundary conditions on $\partial\Omega$.
  This is the problem that will be analyzed in Section \ref{sec:polytopes}.
\end{remark}
%%%%%%%%%%%%%%%%%%%%%%%%%%%%%%%%%%%%%%%%%%%%%%%%%%%%%%%%%%%%%%%%%%%%%%%%%%%%%%
\subsection{Data-to-Solution Map $\cS$}
\label{sec:Dat2Sol}
%%%%%%%%%%%%%%%%%%%%%%%%%%%%%%%%%%%%%%%%%%%%%%%%%%%%%%%%%%%%%%%%%%%%%%%%%%%%%%
The definition \eqref{eq:cDab} and the Lax-Milgram Lemma imply 
the unique solvability of \eqref{eq:problem} for inputs from $\cDab$.
Hence, 
the (nonlinear) data-to-solution operator
\begin{equation}
  \label{eq:solop}
  \cS :
  \begin{cases}
    \cDab \to Y   \\
    a \mapsto u^a,
  \end{cases}
\end{equation}
is well-defined, with image $\cU = \cS(\cDab)\subset Y$.

It follows from \cite[Lemma B.1]{MS2023} that,
% in the setting [Lin],
for any $0<\beta<\alpha<\infty$, the map $\cS$ is Lipschitz in $\cDab$, 
i.e., that there exists a constant $\Clip>0$ (depending on $\alpha, \beta, f$)
such that, for all $a_1,a_2\in\cDab$,
\begin{equation}
  \label{eq:S-Lip}
  \| \cS(a_1) - \cS(a_2) \|_Y \leq  \Clip\| a_1 - a_2 \|_X.
\end{equation}
For compact subsets $ \cKab \subset \cDab \subset X$, 
the sets $\cU = \cS({\cKab})$ are compact in $Y$.
%%%%%%%%%%%%%%%%%%%%%%%%%%%%%%%%%%%%%%%%%%%%%%%%%%%%%%%%%%%%%%%%%%%%%%%%%%%%%%%5
\subsection{Galerkin Discretization}
\label{sec:GalDis}
%%%%%%%%%%%%%%%%%%%%%%%%%%%%%%%%%%%%%%%%%%%%%%%%%%%%%%%%%%%%%%%%%%%%%%%%%%%%%%%%
To build numerically accessible neural operators that approximate
the coefficient-to-solution map $\cS$, 
we need to restrict $\cS$ to finite-parametric outputs. 
One way to achieve this is via Galerkin projection.

For any $ Y_N \subset Y$, 
$N$-dimensional subspace of $Y$,
we introduce the 
\emph{approximate data-to-solution operator} $\cS^{Y_N} : \cDab \to Y_N$ 
via
\begin{equation*}
  \cS^{Y_N} : v \mapsto u^v_N\quad\text{where $u^v_N\in Y_N$ is defined by }
  \frab(v;u^v_N, w_N) = \langle f, w_N\rangle, \; \forall w_N\in Y_N.
\end{equation*}
The Lipschitz bound in \eqref{eq:S-Lip} also holds true for $\cS^{Y_N}$,
\emph{with Lipschitz constant independent of $N$}.

The coefficient-to-solution operator in the abstract setting 
\eqref{eq:bilinear-cont} -- \eqref{eq:problem} 
is bounded and the bilinear form is continuous, 
in the norm $\| \cdot \|_Y$.
\begin{lemma} 
\label{lem:Coef2Sol}
For \eqref{eq:Defb}--\eqref{eq:problem},
  the following bounds hold:
  \begin{equation}
    \label{eq:Snorm}
    \| \cS \|_{L^\infty(\cDab; Y)} \leq \frac{\|f\|_{Y'}}{\alpha-\beta}.
  \end{equation}
  and, for all $a\in\cDab$ and all $u, v\in Y$, with $a_0\in \cDab$
  as in \eqref{eq:a0}
  \begin{equation}
    \label{eq:bound-bilinear}
    \left|\frab(a-\alpha a_0; u, v)\right|\leq \beta \|u\|_Y \|v\|_Y.
  \end{equation}
  \end{lemma}
  \begin{proof}
    The first statement is classical; for the second, bound
    \begin{align*}
      2\left|\frab(a-\alpha a_0; u, v)\right| &=
      \left|
      \frab(a-\alpha a_0; u-v, u-v)
      -
      \frab(a-\alpha a_0; u, u)
      -
      \frab(a-\alpha a_0; v, v)
      \right|
                                                \\
      \overset{\eqref{eq:cDab}}&{\leq}
                                 \beta\left| \frab(a_0; u-v, u-v )
                                 -
                                 \frab(a_0; u, u)
                                 -
                                 \frab(a_0; v, v)
                                 \right|
      \\
      \overset{\eqref{eq:bilinear-cont}}&{ \leq} 2 \beta  \| a_0 \|_X \|u\|_Y\|v\|_Y
                                          = 2 \beta   \|u\|_Y\|v\|_Y.
    \end{align*}
  \end{proof}
\section{Richardson iterations}
\label{sec:it}
We introduce and analyze the iterative scheme that we will
exploit theoretically to derive bounds on the size of the NN in the
approximator network $\cA$ of neural operators $\calG$ in \eqref{eq:G=RAE} 
emulating the data-to solution map $\cS$
of the problem described in Section \ref{sec:PrbForm}.
In this section,
we denote
by $\hPsi_N =\{\hpsi_1, \dots, \hpsi_N\}$
a set of 
orthonormal functions in $Y$ (with respect to the scalar product of $Y$),
and by $Y_N = \spn\hPsi_N$.
\subsection{Variational Formulation}
\label{sec:RichVarFrm}
The Richardson iteration
(see, e.g., \cite{Orszag1980})
is obtained by iteratively solving, for $k = 0, 1, 2, \dots$: 
find $\uvNkpone\in Y_N$ such that
\begin{equation}
  \label{eq:iteration}
   \begin{multlined}
  \frab(a_0;\uvNkpone, w_N) \\
  = \frab(a_0;\uvNk, w_N)-\frac{1}{\alpha}\left(\frab(v;\uvNk, w_N) - (f, w_N)  \right) 
  \quad \forall w_N\in Y_N.
\end{multlined}
\end{equation}
%
%%%%%%%%%%%%%%%%%%%%%%%%
We recall classical results on the  convergence and stability of the iterates,
to derive explicit constants in our setting.
\begin{proposition}
  \label{prop:iterative}
  Let $0<\beta<\alpha<\infty$ and choose 
  $u^{v, (0)}_N = \frac{1}{\alpha}\cS^{Y_N} (a_0)$.
  For any $v\in \cDab$, let $\uvNk$ be the iterates obtained
  with \eqref{eq:iteration}.
  Then, for all $k\in \N$, %and all $v\in \cDab$,
\begin{equation}
    \label{eq:iterative-bounds-a}
    \| \uvNk - \cS^{Y_N}(v) \|_Y \leq \frac{1}{\alpha-\beta} \left(\frac{\beta}{\alpha}  \right)^{k+1}\| f\|_{Y'}
\end{equation}
and
\begin{equation}
    \label{eq:iterative-bounds-a2}
  \| \uvNk\|_Y \leq \frac{1}{\alpha-\beta} \|f\|_{Y'}.
\end{equation}
\end{proposition}
%%%%%%%%%%%%%%%%%%%%%%%%
\begin{proof}
  We start by proving \eqref{eq:iterative-bounds-a}.
  For all $v\in \cDab$, from \eqref{eq:iteration} we have that
  \begin{align*}
    \| \uvNkpone - \cS^{Y_N}(v) \|_Y
    &\leq \sup_{ w_N\in Y_N : \|w_N\|_Y=1} \left| \frab(a_0- v/\alpha; \uvNk - \cS^{Y_N}(v), w_N) \right|
    \\
    & \leq \frac{\beta}{\alpha} \| \uvNk - \cS^{Y_N}(v) \|_Y.
  \end{align*}
  Iterating the inequality above, hence,
  \begin{equation}
    \label{eq:Richardson-err}
    \| \uvNk - \cS^{Y_N}(v) \|_Y \leq \left( \frac{\beta}{\alpha} \right)^k \| u^{v, (0)}_N - \cS^{Y_N}(v) \|_Y.
  \end{equation}
  Noting in addition that
  \begin{align*}
    \| \cS^{Y_N}(v) -  \frac{1}{\alpha}\cS^{Y_N}(a_0)  \|_{Y}
    &\leq \sup_{w\in Y_N: \|w\|_Y=1} \left| \frab(a_0;\cS^{Y_N}(v) -  \cS^{Y_N}(a_0)/\alpha , w ) \right|\\
    & = \sup_{w\in Y_N: \|w\|_Y=1} \left| \frab(a_0  - v/\alpha;\cS^{Y_N}(v)  , w ) \right|\\
    \overset{\eqref{eq:bound-bilinear}}&{\leq} 
    \frac{\beta}{\alpha} \| \cS^{Y_N}(v) \|_Y
    \overset{\eqref{eq:Snorm}}{\leq} \frac{\beta}{\alpha}\frac{\|f\|_{Y'}}{\alpha-\beta},
  \end{align*}
  we obtain \eqref{eq:iterative-bounds-a}.
  Finally, by the same argument,
  \begin{equation*}
    \| \uvNkpone \|_Y \leq \frac{\beta}{\alpha}\|\uvNk\|_Y  +  \frac{1}{\alpha}\|f\|_{Y'}
  \end{equation*}
  and, since $\|u^{v,(0)}_N\|_Y = \|f\|_{Y'}/\alpha$,
  \begin{equation*}
    \| \uvNk \|_Y \leq \frac{1}{\alpha}\sum_{j=0}^{k+1} \left( \frac{\beta}{\alpha} \right)^{j}\|f\|_{Y'}\leq \frac{1}{\alpha-\beta}\|f\|_{Y'}
  \end{equation*}
  which is \eqref{eq:iterative-bounds-a2}.
\end{proof}

\begin{remark} 
\label{rem:RichItRef}
  Richardson iterations or similar techniques have been 
  used as an approximation tool, e.g., 
  in the context of the NN-based emulation of
  operators \cite{Kovachki2021,Kutyniok2022} and for tensor rank
  bounds \cite{Kressner2016}. 
  We have presented the
  iteration here for completeness and to obtain the precise
  estimates needed for the following.
\end{remark}
%%%%%%%%%%%%%%%%%%%%%%%%%%%%%%%%%%%%%%%%%%%%%%%%%%%%%%%%%%%%%%%%%%%%%%%%%%%%%%%%%%%%%%
\subsection{Matrix Form}
\label{sec:RichMatFrm}
%%%%%%%%%%%%%%%%%%%%%%%%%%%%%%%%%%%%%%%%%%%%%%%%%%%%%%%%%%%%%%%%%%%%%%%%%%%%%%%%%%%%%%
Our theoretical bounds on the size of the approximation network
will be based on the matrix form of \eqref{eq:iteration}.
We denote by 
$\hbff_N$ the vector with
entries $\left[ \hbff_N \right]_i = \langle f, \hpsi_i\rangle$ 
and by 
$\hbB_N^{v}$ the matrix with 
entries
\begin{equation}
  \label{eq:Bmat}
\left[\hbB^v_N\right]_{ij} = \frab(v;\hpsi_j, \hpsi_i), \quad\text{for }i, j \in \{1, \dots,
N\}.
\end{equation}
% and $v\in L^\infty(\Omega)$.
Applying the analysis operator
\begin{equation}\label{eq:AnOp}
  \analysis^{\hPsi_N}:
  \begin{cases}
  Y_N \to \R^{N} \\ u^{v, (k)}_N \mapsto \bck_N \;,
  \end{cases}
\end{equation}
to both sides of \eqref{eq:iteration}, 
the Galerkin equation \eqref{eq:iteration} takes the algebraic form
\begin{equation}
  \label{eq:iteration-alg-z}
  \bckpone_N = \bck_N - \frac{1}{\alpha}\left( \hbB_N^{v}\bck_N - \hbff_N\right).
\end{equation}
From \eqref{eq:iterative-bounds-a2}
it follows that, for all $k\in \N$,
\begin{equation}
    \label{eq:iterative-bounds-b}
 \|   \bck_N \|_{\ell^2}  = \| u_N^{v,(k)} \|_Y 
 \leq 
 \frac{1}{\alpha-\beta} \|f\|_{Y'}.
\end{equation}
\section{General bounds on the approximation network}
\label{sec:GenBounds}
%
%%%%%%%%%%%%%%%%%%%%%%%%%%%%%%%%%%%%%%%%%%%%%%%%

As in the previous section, we denote also here,
by
$\hPsi_N = \{\hpsi_1, \dots, \hpsi_N\}$, a set of orthonormal
functions in $Y$ and by $Y_N = \spn\hPsi_N$.
\subsection{Approximation of the iterative scheme with NNs}
\label{sec:approximation-op}
%%%%%
We consider here the approximation part $\cA$ (see Section \ref{sec:neural-op-intro})
of the neural operator $\cG$ in \eqref{eq:G=RAE} 
and derive estimates exploiting the iteration in Section \ref{sec:it}. 
\begin{lemma}
  \label{lemma:Phiinput}
  Let $N, M\in \N$, $\alpha\in\R$ and $\Xi = \{\xi_1, \dots, \xi_M\}\subset X$. 
  Let $\{\hpsi_1, \dots, \hpsi_N\}\subset Y$ be orthonormal in $Y$ and
  let the matrix $\hbB^v_N$ be defined as in
  \eqref{eq:Bmat}. 
  There exists a depth one NN $\Phiinput$ such that for all $\by \in \R^M$
  \begin{equation*}
    \bRealiz(\Phiinput)(\by) = \vec(\Id_N- \alpha^{-1}\hbB^v_N), 
  \end{equation*}
  with $v = \sum_{k=1}^M y_k \xi_k$. Furthermore,
  $\bRealiz(\Phiinput):\R^{M}\to\R^{N^2}$ and $\size(\Phiinput)\leq N^2M + N^2$. 
\end{lemma}
\begin{proof}
  The linearity of the mappings $\by\mapsto  v$ and $v\mapsto \hbB^v_N$ implies
  the existence of the NN, which is a just a linear transformation from $\R^M\to\R^{N^2}$.
  In particular,
  \begin{equation*}
    \Phiinput = \left( (-\alpha^{-1} \left[ \vec\hbB_N^{\xi_1} \dots \vec\hbB_N^{\xi_M} \right], \vec(\Id_N)) \right),
  \end{equation*}
  where  $\left[ \vec\hbB_N^{\xi_1} \dots \vec\hbB_N^{\xi_M} \right]$ is the
  matrix with columns $\vec\hbB_N^{\xi_k}$, $k=1, \dots, M$.
\end{proof}

\begin{lemma}
  \label{lemma:Phistep}
  Let $0<\beta<\alpha<\infty$ and $\rho=\ReLU$.
  Let $\{\hpsi_1, \dots, \hpsi_N\}\in Y^N$ be orthonormal in $Y$ and
  let the matrix $\hbB^v_N$ be defined as in
  \eqref{eq:Bmat}.
  Then, there exists $C>0$ such that, 
  for all $N\in \N$, all $v\in\cDab$, all $\epsilon\in (0,1)$, and for all $Z>0$
  exists a NN $\Phistep_{N, Z, \epsilon}$ 
  such that $\bRealiz_\rho(\Phistep_{N, Z, \epsilon}): \R^{N(N+1)} \to \R^N$ and, for all $\bx\in\R^{N}$ with $\| \bx \|_{\ell^2}\leq Z$,
  \begin{equation*}
  \sup_{v\in\cDab}
  \| \bRealiz_{\rho}(\Phistep_{N, Z, \epsilon})
  \left(\vec(\Id_N - \alpha^{-1}\hbB^{v}_N), \bx\right) %\\
  -\left[\bx - \alpha^{-1}(\hbB^{v}_N \bx - \hbff_N)  \right]\|_{\ell^2}
  \leq \epsilon,
  \end{equation*}
  with $\rho=\ReLU$, and 
  \begin{align*}
  \depth(\Phistep_{N, Z, \epsilon})&\leq C \left( \left| \log\epsilon  \right| + \log N +
      \log(\max(Z, 1)) \right),\\
  \size(\Phistep_{N, Z, \epsilon})&\leq CN^2\left( \left| \log\epsilon  \right| + \log N +
       \log(\max(Z, 1))\right).
  \end{align*}
  \end{lemma}
  \begin{proof}
    The statement follows from \cite[Proposition 3.7]{Kutyniok2022}
    since
   $\| \Id_N - \alpha^{-1}\hbB^{v}_N\|_2 \leq \beta/\alpha<1$ 
   and since $\alpha^{-1}\hbff_N$ 
   can be added as a bias in the output layer.
  \end{proof}
  \begin{lemma}
    \label{lemma:Phiit}
    Let $0<\beta<\alpha<\infty$ and $\rho=\ReLU$. 
Let $\{\hpsi_1, \dots, \hpsi_N\}\in Y^N$ and
let the matrix $\hbB^v_N$ and the vectors $\bck$ be defined as in
Section \ref{sec:RichMatFrm}.
    Then, 
    for all $\epsilon\in(0,1)$, all $v\in\cDab$,
    all $K\in \N$, and all $N\in \N$, 
    there exists a NN $\Phiit_{N, K, \alpha, \beta, \epsilon}$ 
    such that
    \begin{equation}
      \label{eq:it-error}
      \sup_{v\in \cDab}
      \| \bRealiz_{\rho}(\Phiit_{N, K,\alpha, \beta, \epsilon})\left(\vec(\Id_N - \alpha\hbB^{v}_N)\right) - \bcK \|_{\ell^2} \leq \epsilon,
    \end{equation}
    with $\rho = \ReLU$, 
    and, for some $C>0$ that depends on $\alpha$, $\beta$, and $f$,
    \begin{equation}
      \label{eq:it-size}
      \begin{aligned}
      \depth(\Phiit_{N, K,\alpha, \beta, \epsilon})& \leq C K \left( \left| \log\epsilon \right|+ \log N  +1\right),\\
      \size(\Phiit_{N, K,\alpha, \beta, \epsilon}) & \leq C K N^2\left( \left| \log\epsilon \right|+ \log N  +1\right).
      \end{aligned}
    \end{equation}
  \end{lemma}
  \begin{proof}
    We start the proof with the observation 
   that a depth one NN exactly represents 
   \begin{equation*}
     \vec(\Id_N-\alpha^{-1}\hbB^{v}_N) \mapsto ((\vec(\Id_N-\alpha^{-1}\hbB^{v}_N), \bc^{(0)}_N),
   \end{equation*}
   where $\bc^{(0)}_N =  \analysis^{\hPsi_N} \left(\cS^{Y_N}(\alpha a_0)  \right)$.

   We apply a $K$-fold iteration of the network introduced in Lemma \ref{lemma:Phistep}.
   Therefore we consider, in the following, 
   the $K$-fold composition of $\Phistep_{N, \tZ, \epsit}$, 
   for the appropriate values of $\tZ$ and $\epsit$. 
   Specifically, we set
   \begin{equation}
     \label{eq:Zepsit}
     \tZ = 1 + \frac{1}{\alpha-\beta}\|f\|_{Y'}, \qquad \epsit = \left( 1- \frac{\beta}{\alpha} \right)\epsilon.
   \end{equation}
   Furthermore, we write 
   \begin{equation*}
   \bA_N^v = \Id_N-\alpha^{-1}\bB_N^{v}, \qquad
   \tbc_N^{v,(0)} = \bc^{(0)}_N
   \end{equation*}
   and, inductively,
   \begin{equation*}
     \tbckpone_N = \bRealiz_{\rho}(\Phistep_{N, \tZ, \epsit})(\vec(\bA_N^v), \tbck_N),\qquad \forall k\in \N_0.
   \end{equation*}
We recall that $\|\bA_N^v\|_2 \leq \beta/\alpha$. 
We estimate the error at the $(k+1)$th hidden layer:
   \begin{align*}
     \| \tbckpone_N - \bckpone_N \|_{\ell^2}
     &\leq
       \| \tbckpone_N - \bA_N^v \bck_N + \hbff_N\|_{\ell^2} + \| \bA^v(\tbck_N - \bck_N) \|_{\ell^2} \\
     &\leq
       \epsit + \frac{\beta}{\alpha} \|\tbck_N - \bck_N \|_{\ell^2}.
   \end{align*}
   Iterating the inequality above, we obtain
   \begin{equation*}
     \| \tbckpone_N - \bckpone_N \|_{\ell^2} \leq \sum_{j=0}^k\left( \frac{\beta}{\alpha} \right)^j \epsit \leq \frac{\alpha}{\alpha-\beta}\epsit= \epsilon,
   \end{equation*}
   which is \eqref{eq:it-error}.

   We still have to verify that $\| \tbck_N \|_{\ell^2}\leq \tZ$. 
   To do this, we use the inequality above and obtain, for all $k\in\N$,
   \begin{align*}
     \| \tbck_N \|_{\ell^2 }
     \leq \|\tbck_N - \bck_N\|_{\ell^2} + \| \bck_N \|_{\ell^2}
     \overset{\eqref{eq:iterative-bounds-b}}
     &{\leq} \epsilon
    % + \left( \frac{\beta}{\alpha} \right)^k
       + \frac{1}{\alpha-\beta}\|f\|_{Y'} \leq \tZ.
   \end{align*}
   To conclude the proof, we estimate the size of the $K$-fold concatenation of
   the networks $\Phistep_{N, \tZ,\epsit}$. 
   From Lemma \ref{lemma:Phistep}, we obtain
   \begin{align*}
     \depth(\Phistep_{N, \tZ,\epsit})& \leq C \left( \left| \log\epsilon \right|+ \log N  +1\right),\\
     \size(\Phistep_{N, \tZ,\epsit}) & \leq C N^2\left( \left| \log\epsilon \right|+ \log N +1 \right),
   \end{align*}
   where the constant $C$ depends additionally on $f$, $\alpha$, and $\beta$.
  \end{proof}
The following result is close to \cite[Theorem 4.3]{Kutyniok2022}, with two main
differences: we need to have an explicit dependence on the size of the parameter
space $M$, and we want to avoid the cubic dependence on $N$ that would result
from using the same proof strategy as in \cite[Theorem 4.3]{Kutyniok2022}.
  \begin{proposition}
    \label{prop:appx}
    Let $\rho=\ReLU$ and $0<\beta<\alpha<\infty$. For all $\epsilon\in(0,1)$,
    $N\in\N$ and $M\in \N$
    there exists a neural network
    $\Phiappx_{N, M, \alpha, \beta, \epsilon}$ 
    such that
    \begin{itemize}
    \item for all $\by\in \R^M$ and all $\Xi_M = \{\xi_1, \dots, \xi_M\}\in
     X^M$ satisfying $\by\cdot \Xi_M \in \cDab$,
    \item for all linearly independent functions
      $\Psi_N = \{ \psi_1, \dots, \psi_N\}\in Y^N$
  \end{itemize}
  the following bound holds
  \begin{equation}
      \label{eq:Phiapprox}
\| \bRealiz_{{\rho}}(\Phiappx_{N, M, \alpha, \beta, \epsilon})(\by)\cdot \Psi_N - \cS^{Y_N}(\by\cdot\Xi_M) \|_Y 
  \leq \epsilon, \qquad\text{with }Y_N = \spn(\Psi_N).
    \end{equation}
    Furthermore, there is a constant $C>0$ (depending on $\alpha,\beta$) 
    such that, for all $N,M,\epsilon$ as above
    \begin{equation}
      \label{eq:appx-size}
      \begin{aligned}
      \depth(\Phiappx_{N, M, \alpha, \beta, \epsilon}) &\leq C \left| \log\epsilon \right| \left( \left| \log\epsilon \right|+ \log N  +1\right),\\
      \size(\Phiappx_{N, M, \alpha, \beta, \epsilon}) &\leq C  N^2\left( \left| \log\epsilon \right|^2+ \log N \left| \log\epsilon \right| + M \right).
      \end{aligned}
    \end{equation}
  \end{proposition}
  \begin{proof}
    Denote the parametric input as $v = \by\cdot \Xi_M$.
\\
\noindent\textbf{Construction of the network.}
    We choose in the Richardson Iteration in Prop.~\ref{prop:iterative} the number $K$ of steps as

    \begin{equation}
      \label{eq:Kchoice}
      K = \left\lceil \frac{\left| \log(\epsilon/2) \right|+ \left| \log\|f\|_{Y'} - \log(\alpha-\beta)\right| } {\left| \log(\beta/\alpha) \right|}\right\rceil .
    \end{equation}
    We will use the operation of \emph{sparse concatenation $\sconc$ of NNs} 
    as introduced in \cite{Petersen2018}: 
    given two NNs $\Phi_1$ and $\Phi_2$, 
    there exists a third NN, denoted $\Phi_1 \sconc \Phi_2$, 
    such that
    $\Realiz(\Phi_1\sconc\Phi_2) = \Realiz(\Phi_1)\circ \Realiz(\Phi_2)$ 
    with
    $\depth(\Phi_1 \sconc \Phi_2)\leq \depth(\Phi_1) + \depth(\Phi_2)$ 
    and 
    $\size(\Phi_1 \sconc \Phi_2)\leq 2\size(\Phi_1) + 2\size(\Phi_2)$. 

    Let us now introduce an orthonormalized basis $\hPsi_N$, and denote by $\bH_N$
    the (invertible) matrix such that
    \begin{equation}
      \label{eq:Hdef}
      \Psi_N = \bH_N^\top \hPsi_N.
    \end{equation}

    Consider next
    the encoding network $\Phi^{\hbB}$ defined in Lemma~\ref{lemma:Phiinput}
    associated with the basis $\hPsi_N$ and
    the network $\Phiit$ in Lemma~\ref{lemma:Phiit}. 
    We introduce the network
    \begin{equation}\label{eq:ApprNet}
      \Phiappx_{N, M, \alpha, \beta, \epsilon} \coloneqq (\bH_N^{-1}, \bzero)\sconc\Phiit_{N, K,\alpha, \beta,  \epsilon/2} \sconc \Phiinput.
    \end{equation}

\noindent\textbf{Bounding the expression error of the network.}
    Let $u^{v, (K)}_N$ be the $K$th iterate obtained through the iterative
    scheme of Section \ref{sec:it} with datum $v$ and in the space $Y_N$.
    We have that
    \begin{equation}
      \label{eq:approx-bound-1}
      \| u^{v, (K)}_N   - \cS^{Y_N}( v )\|_Y \leq \frac{1}{\alpha-\beta}\left( \frac{\beta}{\alpha} \right)^{K+1} 
      \| f \|_{Y'}\overset{\eqref{eq:Kchoice}}{\leq} \frac{\epsilon}{2}
    \end{equation}
    From
    Lemma  \ref{lemma:Phiit}, writing $\bcK = \analysis^{\hPsi_N}u_N^{v, (K)}$,
    we have
    \begin{align}
      \label{eq:approx-bound-2}
      \begin{aligned}
      \| \bRealiz(\Phiappx_{N, M, \alpha, \beta, \epsilon})(\by)\cdot\Psi_N -  u_N^{v, (K)}\|_{Y}
        &= \| \left( \bH_N\bRealiz(\Phiappx_{N, M, \alpha, \beta, \epsilon})(\by)  - \bcK_N\right)\cdot \hPsi_N\|_{Y}
        \\ &= \| \bH_N\bRealiz(\Phiappx_{N, M, \alpha, \beta, \epsilon})(\by) - \bcK_N\|_{\ell^2}
        \\ &
        = \| \bRealiz(\Phiit_{N, K,\alpha, \beta,  \epsilon/2} \sconc \Phiinput)(\by) - \bcK\|_{\ell^2}
             \\ &
             \leq \frac{\epsilon}{2}.
      \end{aligned}
    \end{align}
    Combining \eqref{eq:approx-bound-1} and \eqref{eq:approx-bound-2} we obtain
    \eqref{eq:Phiapprox}.

    \noindent\textbf{Depth and size estimate.} The bounds \eqref{eq:appx-size} 
    on the depth and size of the approximator network $\Phiappx_{N, M, \alpha, \beta, \epsit}$ 
    follow directly from Lemmas \ref{lemma:Phiinput} and \ref{lemma:Phiit}.
  \end{proof}
  \begin{remark}
    \label{remark:appx-RNN}
    % By the same argument as in Remark \ref{rem:it-RNN}, 
    A part of
    the approximator network constructed in Proposition \ref{prop:appx} 
    has a recurrent structure, 
    being $\mathcal{O}(\left|\log \epsilon\right|))$ many 
    sparse concatenations of the network built in Lemma~\ref{lemma:Phiit}. 
    If this subnetwork is viewed as a RNN, 
    the bound on the depth of $\Phiappx_{N, M, \alpha, \beta, \epsilon}$
    can be divided by a factor $\left| \log\epsilon \right|$.
  \end{remark}
%%%%%%%%%%%%%%%%%%%%%%%%%%%%%%%%%%%%%%%%%%%%%%%%%%%%%%%%%%%%%%%%%%%%%%%%%%%%%%%%%%%%%%%%%%5
\subsection{Bounds based on $N$-widths of the input and holomorphy} 
\label{sec:OpExpRt}
%%%%%%%%%%%%%%%%%%%%%%%%%%%%%%%%%%%%%%%%%%%%%%%%%%%%%%%%%%%%%%%%%%%%%%%%%%%%%%%%%%%%%%%%%%%%
Kolmogorov $N$-widths of solutions sets of elliptic PDEs in polytopes, 
given a certain regularity of the input data $\cK\subset X$,
can be bounded in essentially two ways:

(a) by estimating, via Lemma~\ref{lemma:Nwidth-holomorphic-general}, 
    the $N$-widths of the solution set $\cU = \cS(\cK)$ under holomorphic map $\cS$,
or  
(b) from the solution regularity in function spaces 
    and piecewise polynomial approximation rates.

We now obtain our first set of expression rate bounds, within the framework (a).
Expression rate bounds obtained within framework (b) will be developed 
in Section~\ref{sec:polytopes} ahead.

The next result (Theorem \ref{th:ON-input-Nwidth}) 
uses the $N$-width bound on solution sets 
under holomorphic maps, 
see Lemma~\ref{lemma:Nwidth-holomorphic-general}.
To this end, we require that 
    $\cS$ can be extended to a holomorphic operator. 
    In particular, 
    we suppose
    that $\alpha$ and $\beta$ are such that the following
    assumption is satisfied.
%%%%%%%%%%%%%%%%%%%%%%%%%%%%%%%%%%%%%%%%%%
\begin{assumption} (Holomorphy of the data-to-solution map)
 \label{assumption:holomorphy}
 The data-to-solution map $\cS$ is real analytic from $\cDab$ to $Y$ and it
 admits a holomorphic and uniformly bounded extension from an open set 
 $O\subset X^{\mathbb{C}} = X\otimes \{1, i\}$ 
 into $Y^{\mathbb{C}}     = Y\otimes \{1, i\}$, 
 with $\cDab \subset O$.
\end{assumption}
Assumption \ref{assumption:holomorphy}
holds for linear, second order divergence form PDEs (see, e.g. \cite{Bacuta2017}).
  We next prove an existence result 
  with quantitative size bounds for the approximation network 
  of a neural operator emulating the coefficient-to-solution map of \eqref{eq:problem}.
  We do this under Assumption \ref{assumption:holomorphy} and 
  requiring 
  only an algebraic decay rate of the Kolmogorov $N$-widths of the input set $\cK$.
  \begin{theorem}
    \label{th:ON-input-Nwidth}
   Let $X$ be a uniformly convex Banach space.
   Let $0<\beta<\alpha<\infty$ and $\cS$ be such that 
   Assumption \ref{assumption:holomorphy} holds.
   Let $\cKab\subset\cDab$ be a subset that is compact
   in the $X$-topology, with algebraic $N$-width decay
   $d_M(\cKab, X) \leq C_\cK M^{-s}$ for all $M\in\N$,
   with suitable constants $C_\cK, s>1$.  
   Let $\rho=\ReLU$.
   \\
   Then, for all $\epsilon\in(0,1)$, 
   there exist $M_\epsilon, N_\epsilon \in \N$
   and a $\ReLU$ neural operator 
   \begin{equation*}
     \cG_\epsilon = \cR_\epsilon \circ \bRealiz_{\rho}(\Phi_\epsilon)\circ \cE_\epsilon
   \end{equation*}
   with continuous encoder
   $\cE_\epsilon : \cKab \to \R^{M_\epsilon}$, 
   linear decoder $\cR_\epsilon: \R^{N_\epsilon}\to Y$, 
   and approximator NN
   $\bRealiz_{\rho}(\Phi_\epsilon): \R^{M_\epsilon}\to\R^{N_\epsilon}$,
   satisfying
   \begin{equation}
     \label{eq:bound-generic}
     \| \cS - \cG_\epsilon \|_{L^\infty(\cK;Y)} \leq \epsilon.
   \end{equation}
   In addition, for all $0 < t< s-1$, there exists $C>0$ such that for all
   $\epsilon\in(0,1)$,  $M_\epsilon \leq
   C\epsilon^{-1/s}$,  $N_\epsilon \leq C\epsilon^{-1/t}$, and
   \begin{equation}
     \label{eq:appx-size-corollary}
     % \begin{aligned}
       \depth(\Phi_{\epsilon}) \leq C  \left( \left| \log\epsilon \right|^2   +1\right),\qquad
       \size(\Phi_{\epsilon}) \leq C  \epsilon^{-2/t - 1/s}.
     % \end{aligned}
   \end{equation}
  \end{theorem}
  \begin{proof}
    We denote, in this proof, by 
$$\Xi_N =\{\xi_1, \dots, \xi_N\} \subset X \;\mbox{and}\; 
  \Psi_N = \{\psi_1, \dots, \psi_N\} \subset Y,
$$
      quasi-optimal sequences in the sense of $N$-widths, i.e., 
      such that for $C_1\geq 1$
      \begin{equation*}
        \dist_X(\cKab, \spn\Xi_M) \leq C_1 d_N(\cKab, X),
      \end{equation*}
      and
      \begin{equation*}
        \dist_Y(\cS(\cKab), \spn\Psi_N) \leq C_1 d_N(\cS(\cKab), Y)
      \end{equation*}
 for all $N\in\N$.

The assumed algebraic decay of the $N$-width $d_N(\cKab, X)$ and compactness of
$\cKab$ ensure that
there exists $M_0\in \N$ and
$\tbeta\in (0,\alpha)$ such that, 
for all $M\geq M_0$,
    \begin{equation*}
      \cE_M (a)\cdot \Xi_M\in \cD_{\alpha, \tbeta}, \qquad \forall a\in \cKab,
    \end{equation*}
    where we have introduced $\cE_M$ such that for all $a\in\cKab$
    \begin{equation*}
      \| \cE_M(a) \cdot \Xi_M -a\|_X = \dist(a, \spn\Xi_M).
    \end{equation*}
Continuity of $\cE_M$ is ensured by the 
uniform convexity of $X$ \cite[Proposition 3.2]{Goebel1984}.

    Next, fix $t\in (0, s-1)$.
    As $\cS^{Y_N}$ is Lipschitz continuous uniformly with respect to $N$ \cite[Lemma B.1]{MS2023}, 
    there exists $\tClip>0$
    that depends on $\alpha, \tbeta$ (hence $\beta$ and $M_0$) and $f$, such that
    \begin{equation}
      \label{eq:LipSXN}
      \sup_{a,b \in \cD_{\alpha, \tbeta}}\| S^{Y_N}(a) - S^{Y_N}(b) \|_{Y}\leq \tClip\|a -b \|_{L^\infty(\Omega)}.
    \end{equation}
    Fix a value of $M_0$ and $\tbeta$ and let
    \begin{equation}
      \label{eq:Meps}
      M_\epsilon = \max\left( M_0, \left\lceil   \left(\frac{\epsilon}{3C_1C_{\cK} \tClip}   \right)^{-1/s}\right\rceil\right),
    \end{equation}
    so that $C_1\tClip d_{M_\epsilon}(\cKab, L^\infty(\Omega))\leq \epsilon/3$.

    By Lemma \ref{lemma:Nwidth-holomorphic-general} and Assumption \ref{assumption:holomorphy},
    furthermore, there exists $\tC_\cU>0$ such that
    \begin{equation}
      \label{eq:Nwidth-U-t}
      \dist_Y(\cS(\cKab), \spn\Psi_N) \leq   C_1\tC_\cU N^{-t}, \qquad \forall N\in\N. 
    \end{equation}
    Let 
    \begin{equation*}
       N_\epsilon = \left\lceil \left( \frac{\alpha-\beta}{\alpha+\beta}\frac{\epsilon}{3C_1\tC_\cU} \right)^{-1/t} \right\rceil,
    \end{equation*}
    and $\cR_\epsilon(\by) = \by\cdot\Psi_{N_\epsilon}$ for all $\by\in\R^{N_\epsilon+1}$.
    Choose, with the notation of Proposition \ref{prop:appx}, 
    $$\Phi_{\epsilon} = \Phiappx_{N_\epsilon, M_\epsilon, \alpha, \tbeta, \epsilon/3}.$$ 
    We write 
    \begin{align*}
      \| \cS(a) - \cG_{\epsilon}(a) \|_Y
      &
        \leq
        \begin{aligned}[t]
          &\| \cS(a) - \cS^{Y_{N_\epsilon}}(a) \|_Y + \| \cS^{Y_{N_\epsilon}}(a) - \cS^{Y_{N_\epsilon}}(\cE_{M_\epsilon}(a)\cdot\Xi_{M_\epsilon}) \|_Y
          \\   &\qquad  + \|\cS^{Y_{N_\epsilon}}(\cE_{M_\epsilon}(a)\cdot\Xi_{M_\epsilon})  - \cG_\epsilon(a)\|_Y
        \end{aligned}
      \\
      &\eqqcolon (I)+  (II) + (III).
    \end{align*}
    We bound the three terms. For term $(I)$, using \eqref{eq:Nwidth-U-t},
    \begin{equation*}
      \sup_{a\in \cD} (I)    \leq \frac{\alpha+\beta}{\alpha-\beta}\dist(\cS(\cKab), Y_{N_\epsilon})\leq \frac{\alpha+\beta}{\alpha-\beta}C_1\tC_{\cU}N_\epsilon^{-t}\leq \frac{\epsilon}{3}.
    \end{equation*}
    Term $(II)$ is estimated as
    \begin{equation*}
      (II) \overset{\eqref{eq:LipSXN}}{\leq} \tClip \| a  - \cE_{M_\epsilon}(a) \cdot \Xi_{M_\epsilon} \|_{L^\infty(\Omega)}  \leq C_1\tClip C_\cK M_\epsilon^{-s} \overset{\eqref{eq:Meps}}{\leq} \frac{\epsilon}{3}.
    \end{equation*}
    Finally, for term $(III)$ we use Proposition \ref{prop:appx} to obtain
    \begin{equation*}
      (III) \leq \frac{\epsilon}{3}.
    \end{equation*}
    This concludes the proof of \eqref{eq:bound-generic}. The bounds on the
    depth and size of $\Phi_\epsilon$ also follow from Proposition \ref{prop:appx}.
  \end{proof}
In the absence of uniform convexity of $X$, 
reflexive Banach spaces exist with discontinuous nearest point projection in $X$
onto $\spn\Xi_M$ \cite{ALBrown74}.

%%%%%%%%%%%%%%%%%%%%%%%%%%%%%%%%%%%%%%%%5
  \subsection{Explicit bounds with bases on the input and output spaces}
  \label{sec:full-op}
%%%%%%%%%%%%%%%%%%%%%%%%%%%%%%%%%%%%%%%%%
  We prove bounds on the size of the approximator network, 
  in the context of the encoder-approximator-decoder architecture \eqref{eq:G=RAE}
  in an abstract setting. 
The goal of this is mostly technical: Proposition \ref{prop:general-ON} 
will be instrumental for 
proving the expression rate bounds in Section \ref{sec:polytopes}.
  \begin{setting}
    \label{set:base}
    Let $0<\beta<\alpha<\infty$ and let $\cKab\subset \cDab$ be a compact subset
    with respect to the $X$-induced topology. 
    Denote $\cU = \cS(\cKab)$. 
    \begin{enumerate}
    \item\label{item:base-conv} 
      For all $M\in\N$,
      encoder maps $\cE_M : \cKab \to \R^M$ and functions $\Xi_M = \{\xi_1, \dots, \xi_M\}\in X^M$ 
      are given such that
      \begin{equation}
        \label{eq:encoder-conv}
        \sup_{a\in \cKab}\| \cE_M (a)\cdot\Xi_M - a \|_{X} \to 0  \qquad \text{as }M\to\infty.
      \end{equation}
    \item $\Psi_N = \{\psi_1, \dots, \psi_N\}\in Y^{N}$ are linearly independent
      and $Y_N =\spn(\Psi_N)$.
    \end{enumerate}
\end{setting}
The following proposition is a technical result to prove
expression rate bounds for coefficient-to-solution maps for
linear, second order, divergence-form PDEs, as specified 
in Setting~\ref{setting:base-polytope} ahead.
Its proof is similar to that of Theorem \ref{th:ON-input-Nwidth}, 
but argues with Proposition \ref{prop:appx}.
It \emph{does not invoke Lemma~\ref{lemma:Nwidth-holomorphic-general}},
i.e., it
does not require the holomorphic extension of $\cG$.

  \begin{proposition}
    \label{prop:general-ON}
    Assume Setting \ref{set:base}. 
    Then, 
    there exists $M_0 \in \N$ such that 
    for all $M\geq M_0$, all $N\in\N$, and all $\epsapp\in (0,1)$,
    there exist an approximator NN $\Phi_{N, M,\epsapp}$ 
    and a neural operator
    \begin{equation*}
      \cG_{N, M, \epsapp}:
        \begin{cases}
        \cDab \to Y \\ 
      a\mapsto \Psi_{N}\cdot \left(\bRealiz_{{\rho}}(\Phi_{N, M, \epsapp}) \circ \cE_M(a)  \right)
        \end{cases}
    \end{equation*}
    with 
    $\rho=\ReLU$ activation, 
    such that
    \begin{equation*}
      \| \cS - \cG_{N, M, \epsapp} \|_{L^\infty(\cKab; Y)} 
      \lesssim 
      \sup_{a\in \cKab}\| \cE_M(a)\cdot \Xi_M - a \|_{L^\infty(\Omega)} + \dist_Y(\cU, Y_N) + \epsapp 
    \end{equation*}
    and, as $\epsapp \to 0$ and $M, N \to \infty$,
    \begin{align*}
      \depth(\Phi_{N, M, \epsapp})  &= \cO\left( \left| \log\epsapp \right| \left( \left| \log\epsapp \right|+ \log N  \right) \right),
      \\
      \size(\Phi_{N, M, \epsapp}) &= \cO\left(   N^2\left( \left| \log\epsapp \right|^2+ \log N \left| \log\epsapp \right| + M \right) \right).
    \end{align*}
  \end{proposition}
  \begin{proof}
    By Item \ref{item:base-conv} of Setting
    \ref{set:base} and by \eqref{eq:encoder-conv}, there exists $M_0\in
    \N$ such that there exists $\tbeta\in (0,\alpha)$ such that, 
     for all $M\geq M_0$,
    \begin{equation*}
      \cE_M (a)\cdot \Xi_M\in \cD_{\alpha, \tbeta}, \qquad \forall a\in \cDab.
    \end{equation*}
    Choose next $\Phi_{N, M, \epsapp} = \Phiappx_{N, M, \alpha, \tbeta, \epsapp}$, 
    where the NN $ \Phiappx_{N, M, \alpha, \tbeta, \epsapp}$ 
    on the right hand side is as in Proposition~\ref{prop:appx}.
    We write 
    \begin{align*}
      &\| \cS(a) - \cG_{N, M, \epsapp, \epsrec}(a) \|_Y
    \\   &\qquad\qquad
        \leq
        \begin{aligned}[t]
        &\| \cS(a) - \cS^{Y_N}(a) \|_Y + \| \cS^{Y_N}(a) - \cS^{Y_N}(\cE_M(a)\cdot\Xi_M) \|_Y
   \\   &\qquad  + \|\cS^{Y_N}(\cE_M(a)\cdot\Xi_M)  - \left( \Realiz(\Phiappx_{N, M, \alpha, \tbeta, \epsapp})  \circ \cE_M \right)(a)\cdot \Psi_N\|_Y
        \end{aligned}
      \\
      &\qquad\qquad\eqqcolon (I)+  (II) + (III).
    \end{align*}
    We bound the four terms. For term $(I)$, by the Lax-Milgram lemma,
    \begin{equation*}
     \sup_{a\in \cKab} (I)   \leq \frac{\alpha+\beta}{\alpha-\beta} \sup_{a\in\cKab}\dist_Y(\cS(a), Y_N).
    \end{equation*}
    For term $(II)$, since $\cS^{Y_N}$ is Lipschitz continuous,
    \begin{equation*}
      (II) \leq \Clip \| a - \cE_M(a)\cdot \Xi_M \|_{L^\infty(\Omega)} 
           \leq \Clip \sup_{a\in \cKab}\|  a  - \cE_M(a)\cdot \Xi_M \|_{L^\infty(\Omega)} ,
    \end{equation*}
    with constant $\Clip$ that depends on $\alpha, \tbeta$ (hence $\beta$ and $M_0$) and $f$.
    By Proposition \ref{prop:appx},
    \begin{equation*}
      (III) \leq \epsapp.
    \end{equation*}
  \end{proof}
%%%%%%%%%%%%%%%%%%%%%%%%%%%%%%%%%%%%%%%%%%%%
\section{Approximation rates for PDEs in polytopes}
\label{sec:polytopes}
%%%%%%%%%%%%%%%%%%%%%%%%%%%%%%%%%%%%%%%%%%%%
We develop the preceding, abstract framework 
for the special case of the model,
linear elliptic PDE \eqref{eq:Ldivagrad} below, in a polytopal domain $\Omega$.
The ensuing arguments 
extend with minor modifications to other linear, divergence-form elliptic PDEs.
We consider here neural operators with the deepOnet architecture introduced in
\cite{lu2020deeponet}, i.e., such that, for $a\in\cD_{\alpha, \beta}$,
\begin{equation}
  \label{eq:deepONet-arch}
  \cG(a) : x\mapsto \bRealiz\left( \Phiappx \right)(a(\bX))\cdot\bRealiz(\Phirec)(x),
\end{equation}
where $\bX$ is a set of points in $\Omega$. We bound the number of points $\bX$,
and the sizes of the NNs $\Phirec$, $\Phiappx$ with respect to the approximation error.
We show algebraic convergence rates for data with finite Sobolev regularity, and
exponential convergence for analytic data.

Let $d\in \{2, 3\}$ and suppose that $\Omega\subset\R^d$ is an open bounded
set with Lipschitz boundary; additional assumptions on $\Omega$ will be made
throughout the section.
We consider in this section 
linear, elliptic, second order, divergence-form PDEs 
as specified in the following setting.
%%%%%%%%%%%%%%%%%%%%%%%%
\begin{setting}
  \label{setting:base-polytope}
  Let $X=L^\infty(\Omega)$, $Y = H_0^1(\Omega)$, and
\begin{equation}
  \label{eq:Ldivagrad}
  L(a) u = -\nabla\cdot(a\nabla u)
\end{equation}
with associated bilinear form 
${\frab}(a; u, v) = (a\nabla u, \nabla v)_{L^2(\Omega)}$. 
The solution operator $\cS : \cDab \to Y$ satisfies
\begin{equation}
  \label{eq:sol-polytope}
 {\frab}(a; \cS(a), v) = \langle f, v \rangle\qquad \forall v\in Y.
\end{equation}
\end{setting}
\begin{remark}\label{rmk:MixBCs}
Setting~\ref{setting:base-polytope}
assumes homogeneous Dirichlet BCs on all of $\partial\Omega$.
All results that follow remain valid for the PDE $L(a)u = f$ 
with homogeneous mixed boundary conditions: there exists a 
partition $\{ \Gamma_D, \Gamma_N \}$ of $\partial\Omega$ 
into open sets such that 
$\partial\Omega = \overline{\Gamma_D \cup \Gamma_N}$, 

with surface measures 
$|\Gamma_D|>0$, $|\Gamma_N| \geq 0$, 
and such that
$$
u|_{\Gamma_D} = 0 \;, \qquad {\nu} \cdot a\nabla u|_{\Gamma_N} = 0 \;.
$$
Here, ${\nu}\in L^\infty(\Gamma_N; \R^d)$  
denotes the outward-pointing unit normal vector to $\partial\Omega$.
In this case, 
\begin{equation}\label{eq:Xmix}
Y = H^1_{\Gamma_D}(\Omega) \coloneqq \{ v\in H^1(\Omega) : \; v|_{\Gamma_D} = 0 \}.
\end{equation}
The data-to-solution map for \eqref{eq:sol-polytope} is still holomorphic,
with the stated (complex extension of) $Y$.
\end{remark}
%%%%%%%%%%%%%%%%%%%%%%%%%%%%%%%%%%%%%%%%%%%%%%%%%%%%%%%%%%%%%%
\subsection{Exponential convergence for analytic data}
\label{sec:ExpCnvAnDat}
In this section, we discuss the case where the data set 
consists of analytic functions in $\overline{\Omega}$.
In this case, we can prove the existence of neural operators
with the architecture \eqref{eq:deepONet-arch}
that converge exponentially to the
solution operator.

Throughout the section, we assume the following analytic data setting.
\begin{setting}
  \label{set:analytic} {[Analytic Data]}
Assume Setting \ref{setting:base-polytope} and that
the domain $\Omega$ is a polygon with a finite number of straight sides if $d=2$
and it is an axiparallel polyhedron if $d=3$.
Finally, 
assume that there exist $0<\beta<\alpha<\infty$ and $A>0$
such that
  \begin{equation}
    \label{eq:analytic-cD}
    \cKab = \left\{ v\in \cDab :   \|  v\|_{W^{m, \infty}(\Omega)} \leq A^{m+1}m!,\, \forall m\in \N\right\}.
  \end{equation}
\end{setting}
\begin{remark}
  By ``axiparallel polyhedron'' we indicate a polyhedron whose edges are
  parallel to the coordinate axes. %This includes, e.g., the Fichera corner domain.
\end{remark}
%%%%%%%%%%%%%%%%%%%%%%%%%%%%%%%%%%%%%%%%%%%%%%%%%%%%%%%%%%%%%%%%%%%%
\subsubsection{Encoder}
\label{sec:encoding-op}
We give an existence result for the 
encoder operator based on point evaluations, with 
quantitative bounds on the dimension of its image.
\begin{lemma}\label{lemma:point-encoder-analytic}
Assume Setting \ref{set:analytic}. 
Then, for all $\epsilon \in (0,1)$, 
exist
$M_\epsilon$ points $\bX = \{\bx_1, \dots, \bx_{M_\epsilon}\}\subset \overline\Omega$, 
and functions
$\Xi = \{\xi_1, \dots, \xi_{M_\epsilon}\} \in L^\infty(\Omega)^{M_\epsilon}$ such that,
\begin{equation*}
  \sup_{a\in\cKab}\| a - a(\bX)\cdot\Xi \|_{L^\infty(\Omega)} \leq \epsilon, 
  \qquad 
  \text{with }a(\bX) = \{a(\bx_1), \dots, a(\bx_{M_\epsilon})\},
\end{equation*}
and $M_\epsilon = \cO(\left| \log\epsilon \right|^{d})$.
\end{lemma}
\begin{proof}
Construct a regular mesh of convex quadrilaterals in $\Omega$ (see Appendix \ref{app:NwdthAn}
for a concrete construction in an arbitrary polytopal domain $\Omega \subset \R^d$).
It implies that 
$M_\epsilon = \cO(\left| \log\epsilon \right|^{d})$ 
and that the points $\bx_1, \dots, \bx_{M_\epsilon}$ 
can be chosen as mapped tensor product Gauss-Legendre-Lobatto nodes,
of order $p = O(|\log \epsilon|)$.
  The functions $\{\xi_1, \dots, \xi_{M_{\epsilon}}\}$ 
  are the corresponding high-order Lagrange basis of the continuous, piecewise polynomial
  functions subject to a partition into convex quadrilaterals in $\Omega$. 
\end{proof}

\subsubsection{Decoder}
\label{sec:decoding-op}
  We start by considering approximation rates for decoders with only $\ReLU$ activation. 
  As a consequence, the realizations of the
  functions in the decoder cannot be exact, continuous piecewise polynomial
  of degree $\geq 2$ \footnote{$\mathbb{P}_1$-Lagrangian FEM on arbitrary regular, simplicial tesselations in any dimension 
  can be exactly represented through ReLU NNs, see \cite{Longo2023}.} functions.
  We recall a result from \cite{MOPS2023}.
  \begin{lemma}
    \label{lemma:rec-analytic}
Assume Setting \ref{set:analytic} and let $\rho = \ReLU$. 
There exists $C>0$ such that for all $\epsilon\in (0,1)$, 
there exists $N\in\N$ satisfying $N\leq C(\left| \log\epsilon \right|^{2d}+1)$ and a NN $\Phirec_N$,
with $\bRealiz_{\rho}(\Phirec_N)\in Y^N$ such that
  \begin{equation*}
    \dist_Y(\cU, \spn(\bRealiz_\rho(\Phirec_N)) \leq \epsilon
  \end{equation*}
  and
  \begin{equation*}
    \qquad
    \depth(\Phirec_N) \leq C(\left| \log\epsilon \right|\log\left|  \log\epsilon\right|+1)
    \qquad
    \size(\Phirec_N) \leq C(\left| \log\epsilon \right|^{2d+1}+1).
  \end{equation*}
  \end{lemma}
\begin{proof}
  This follows from Lemma \ref{lemma:nwidth-exp} and the construction
  of the approximating networks in the proof of \cite[Theorem 4.2]{MOPS2023}.
  In the notation of the latter, the NNs $\Phi_{\epsilon, c}$ are linear
  combinations of $N = N_{\mathrm{1d}}^d$ functions,
  with $N_{\mathrm{1d}}\lesssim 1+\left| \log\epsilon \right|^2$. The bounds on
  the size of the networks are those of \cite[Theorem 4.2]{MOPS2023}. The
  bound on the approximation error follows as in \cite[Theorem 4.3]{MOPS2023}.
\end{proof}

\begin{remark}
  \label{remark:relu-dec}
  The networks introduced in \cite{MOPS2023} are based on tensor product,
  piecewise polynomial approximants. This choice implies that NNs as those in 
  Lemma \ref{lemma:rec-analytic} can approximate functions that have singularities
  on the whole boundary $\partial\Omega$, as, e.g., the solutions of fractional
  Poisson problems \cite{Faustmann2025}. On the other side,
  it also implies that for the equation
  considered here, whose solutions have singularities only along corners and
  (when $d=3$) edges, the exponent $2d$ of Lemma \ref{lemma:rec-analytic} is suboptimal 
  (compare with the exponent $d+1$ in Lemma \ref{lemma:decoder-ReLU2-analytic} below).
\end{remark}
We consider approximation rates for decoders with 
$\ReLU$ and $\ReLU^2$ activation functions. Since proving expression rates for decoders
with these activations is out of the scope of this paper and such results are
only available for polygons, we only consider the two-dimensional case here.
\begin{lemma}
  \label{lemma:decoder-ReLU2-analytic}
  Assume Setting \ref{set:analytic} with $d=2$. Then, there exists $C>0$ such
  that for all $\epsilon\in (0,1)$, there exists $N\in\N$ satisfying $N\leq
  C(\left| \log\epsilon \right|^{d+1}+1)$,
  a NN $\Phirec_N$
  with $\bRealiz_{\brho}(\Phirec_N)\in Y^N$,
  and $\brho\in \{\ReLU, \ReLU^2\}^{\depth(\Phirec_N)-1}$,
  such that
  \begin{equation*}
    \dist_Y(\cU, \spn(\bRealiz_\brho(\Phirec_N)) \leq \epsilon
  \end{equation*}
  and
  \begin{equation*}
    \qquad
    \depth(\Phirec_N) \leq C(\log\left|  \log\epsilon\right|+1)
    \qquad
    \size(\Phirec_N) \leq C(\left| \log\epsilon \right|^{d+1}+1).
  \end{equation*}
\end{lemma}
\begin{proof}
  This follows from Lemma \ref{lemma:nwidth-exp} and 
  \cite[Proposition 4.5 and Theorem 5.1]{Opschoor2024}.
\end{proof}

%
%%%%%%%%%%%%%%%%%%%%%%%%%%%%%%%%%%%%%%%%%%%% 
\subsubsection{Full operator}
\label{sec:fullONanalytic}
%%%%%%%%%%%%%%%%%%%%%%%%%%%%%%%%%%%%%%%%%%%%
%
We combine Proposition \ref{prop:general-ON} with the results of Sections
\ref{sec:Nwidth-polytope}, \ref{sec:encoding-op}, \ref{sec:decoding-op} to
show expression rates for the neural operator approximation of solution sets of
\eqref{eq:Ldivagrad} 
in polytopal domains $\Omega$, with source terms $f$ which 
are analytic in $\overline{\Omega}$.
\begin{theorem}
  \label{th:fullONanalytic}
 Assume Setting \ref{set:analytic} and $\rho=\ReLU$. 

 Then, for all $\epsilon \in (0,1)$, 
 there exist dimensions $N(\epsilon), M(\epsilon) \in \N$,
 points $\bX_\epsilon = \{\bx_1, \dots, \bx_{M}\}\subset\overline{\Omega}$
 and operators 
 \begin{equation*}
   \cG_\epsilon:
     \begin{cases}
     \cKab\to Y \\
   a\mapsto  \left(\bRealiz_{\rho}(\Phiappx_{\epsilon}) \left( a(\bX_\epsilon) \right)  \right)\cdot\bRealiz_{\rho}(\Phirec_{\epsilon}) 
     \end{cases}
 \end{equation*}
 where 
 $\bRealiz_{\rho}(\Phirec_{\epsilon}) : \Omega \to \R^N$
 and
 $\bRealiz_{\rho}(\Phiappx_\epsilon) : \R^M\to \R^N$,
 satisfying
 \begin{equation*}
   \| \cS - \cG_{\epsilon} \|_{L^\infty(\cKab, Y)} \leq \epsilon.
 \end{equation*}
In addition, as $\epsilon\to 0$,
 $M(\epsilon)  = \cO(\left|\log\epsilon\right|^d)$, 
 $N(\epsilon)  = \cO(\left| \log\epsilon \right|^{2d})$,
 and
\begin{align*}
   &\depth(\Phiappx_{\epsilon})  = \cO\left( \left| \log\epsilon \right|^2 \right),\qquad
   \size(\Phiappx_{\epsilon}) = \cO\left(   \left| \log\epsilon \right|^{5d} \right),
   \\
   &\depth(\Phirec_{\epsilon}) = \cO\left( \left| \log\epsilon \right|\log\left| \log\epsilon \right|\right),
   \qquad
   \size(\Phirec_{\epsilon}) =\cO\left( \left| \log\epsilon \right|^{2d+1} \right).
 \end{align*}
\end{theorem}
\begin{proof}
  We specify the components in the architecture \eqref{eq:G=RAE}.

  \noindent
  1. Decoder $\cR$: 
  From Lemma \ref{lemma:rec-analytic}, 
  it follows that for all %$N\in\N$ and
  $\epsrec\in (0,1)$ there exists a NN $\Phirec_{N}$ 
  of output dimension $N \lesssim \left|\log\epsrec\right|^{2d}$ 
  such that
  \begin{equation*}
    \dist_{H^1_0(\Omega)}(\cU, \spn(\bRealiz_\rho(\Phirec_N)) \leq \epsrec
  \end{equation*}

\noindent
  2. Encoder $\cE$:
  by Lemma \ref{lemma:point-encoder-analytic}, 
  there exist $\cE_M:  a\mapsto a(\bX_\epsilon) \in\R^M$ and $\Xi_M \in
  L^\infty(\Omega)^M$ such that
  \begin{equation*}
     \sup_{a\in \cKab}\| \cE_M(a) \cdot \Xi_M -a \|_{L^\infty(\Omega)} \leq \epsilon/3,
  \end{equation*}
  if $M \simeq \left| \log\epsilon \right|^d$.
  
 \noindent
 3. Approximator $\cA$:  
  Using Proposition \ref{prop:general-ON} with
  \begin{equation*}
    \Psi_N = \bRealiz_\rho(\Phirec_N), \qquad \epsrec = \epsilon/3,\qquad \epsapp=\epsilon/3
  \end{equation*}
  concludes the proof.
\end{proof}

When decoders with both $\ReLU$ and $\ReLU^2$ activations are considered, 
the preceding result can be improved. 
As in Lemma \ref{lemma:decoder-ReLU2-analytic}, 
we consider here only the case $d=2$.
\begin{theorem}
  \label{th:fullONanalytic-ReLU2}
 Assume Setting \ref{set:analytic} with $d=2$, $\rho=\ReLU$, and $\brho=\{\ReLU, \ReLU^2\}$. 

 Then, for all $\epsilon \in (0,1)$, 
 there exist dimensions $N(\epsilon), M(\epsilon) \in \N$,
 points $\bX_\epsilon = \{\bx_1, \dots, \bx_{M}\}\subset\overline{\Omega}$
 and operators
 \begin{equation*}
   \cG_\epsilon:
     \begin{cases}
     \cKab\to Y \\
   a\mapsto  \left(\bRealiz_{\rho}(\Phiappx_{\epsilon}) \left( a(\bX_\epsilon) \right)  \right)\cdot\bRealiz_{\brho}(\Phirec_{\epsilon}) 
     \end{cases}
 \end{equation*}
 where $\brho\in \{\ReLU, \ReLU^2\}^{\depth(\Phirec_{\epsilon})-1}$,
 $\bRealiz_{\brho}(\Phirec_{\epsilon}) : \Omega \to \R^N$,
 and
 $\bRealiz_{\rho}(\Phiappx_\epsilon) : \R^M\to \R^N$,
 satisfying
 \begin{equation*}
   \| \cS - \cG_{\epsilon} \|_{L^\infty(\cKab, Y)} \leq \epsilon.
 \end{equation*}
In addition, as $\epsilon\to 0$,
 $M(\epsilon)  = \cO(\left|\log\epsilon\right|^d)$, 
 $N(\epsilon)  = \cO(\left| \log\epsilon \right|^{d+1})$,
 and
\begin{align*}
   &\depth(\Phiappx_{\epsilon})  = \cO\left( \left| \log\epsilon \right|^2 \right),\qquad
   \size(\Phiappx_{\epsilon}) = \cO\left(   \left| \log\epsilon \right|^{3d+2} \right),
   \\
   &\depth(\Phirec_{\epsilon}) = \cO\left( \log\left| \log\epsilon \right|\right),
   \qquad
   \size(\Phirec_{\epsilon}) =\cO\left( \left| \log\epsilon \right|^{d+1} \right).
 \end{align*}
\end{theorem}
\begin{proof}
  The proof is the same as that of Theorem \ref{th:fullONanalytic}, with Lemma
  \ref{lemma:decoder-ReLU2-analytic} replacing Lemma \ref{lemma:rec-analytic}.
\end{proof}

\subsection{Algebraic convergence rates for data with finite regularity in polygons}
\label{sec:algconv}
In this section we consider the case where the data is contained in a ball of
finite Sobolev regularity and the domain is a plane polygon. 
\begin{setting}
  \label{set:finite-reg}
  Assume Setting \ref{setting:base-polytope} and, in addition,
  let $d=2$ and the domain $\Omega$ be a polygon with a finite number of sides.
  Assume that there exist $0<\beta<\alpha<\infty$, $R>0$, and $m\in \N$, $m\geq 2$ such that
    \begin{equation*}
     \cKabm = \cDab \cap \left\{ v\in W^{m, \infty}(\Omega) : \| v \|_{W^{m, \infty}(\Omega)}\leq R\right\}.
    \end{equation*}
\end{setting}
\subsubsection{Encoder}
\label{sec:Encod}
\begin{lemma}
  \label{lemma:enc-finite-1}
  Assume Setting \ref{set:finite-reg}. Then, for all $\epsilon \in (0,1)$, there exist
  $M_\epsilon$ points $\bX = \{\bx_1, \dots, \bx_{M_\epsilon}\} \subset \overline{\Omega}$ and functions
  $\Xi = \{\xi_1, \dots, \xi_{M_\epsilon}\} \in L^\infty(\Omega)^{M_\epsilon}$ 
  such that
  \begin{equation*}
    \sup_{a\in\cKabm}\| a - a(\bX)\cdot\Xi \|_{L^\infty(\Omega)} \leq \epsilon
    \;\; \mbox{ and } \;\; 
    M_\epsilon = \cO(\epsilon^{-2/m}) \;.
  \end{equation*}
\end{lemma}
\begin{proof}
The assertion
follows from (known) rates of approximation by piecewise polynomial functions 
on a sequence $\{ \cT_n \}_{n\geq 1}$ of regular, 
quasi-uniform triangulations $\cT_n$ of $\Omega$.
\end{proof}
%%%%%%%%%%%%%%%%%%%%%%%%%%%%%%%%%%%%%%%%%%%%%%%%%%%%%
\subsubsection{Decoder}
\label{sec:DecOp}
%%%%%%%%%%%%%%%%%%%%%%%%%%%%%%%%%%%%%%%%%%%%%%%%%%%%%
The solution set $\cU$ in Setting~\ref{set:finite-reg} is included in a 
finite-order, corner-weighted Sobolev space.
Approximation rates for
the decoder are a consequence of the results in \cite{Opschoor2024}.
\begin{lemma}>
  \label{lemma:dec-finite}
  Assume Setting \ref{set:finite-reg}. 

  Then, there exists $C>0$ such that
  for all $\epsilon\in (0,1)$, 
  there exist $N\in \N$ satisfying $N\leq C\epsilon^{-2/m}$
  and a NN $\Phirec_{N}$, 
  such that
  \begin{equation*}
    \dist_Y(\cU, \spn(\bRealiz_\brho(\Phirec_N))) \leq \epsilon
  \end{equation*}
  with $\brho\in \{\ReLU, \ReLU^2\}^{\depth(\Phirec_{\epsilon})-1}$
  and
  \begin{equation*}
    \depth(\Phirec_{N}) \leq C,
    \qquad
    \size(\Phirec_{N}) \leq C \epsilon^{-4/m}.
  \end{equation*}
\end{lemma}
\begin{proof}
  By the argument of the proof of Lemma \ref{lemma:FinReg2d}, there exists a
  finite element space $V_h^m$ of continuous, piecewise polynomial functions of total 
  degree $m$, defined on a regular, simplicial partition $\cT$ of $\Omega$, 
  such that
  \begin{equation*}
    \sup_{a \in \cKabm}
    \inf_{v_h \in V^m_h}
    \| u^a - v_h \|_{H^1(\Omega)}
    % this error bound is valid in H1, also w.o. zero BCs
    \leq
    \tC
    (\dim(V^m_h))^{-m/2}
    \;,
  \end{equation*}
  with $\tC$ that depends on $m, R, t, \Omega$.
  We conclude using \cite[Proposition 4.6]{Opschoor2024}.
  \end{proof}
%%%%%%%%%%%%%%%%%%%%%%%%%%%%%%%%%%%%%%%
\subsubsection{Full operator}
\label{sec:OpN}
%%%%%%%%%%%%%%%%%%%%%%%%%%%%%%%%%%%%%%%
\begin{theorem}
 \label{th:fullONfinite-reg}
 Assume Setting \ref{set:finite-reg}. 
Then, for all $\epsilon \in (0,1)$, 
 there exist dimensions $N(\epsilon), M(\epsilon) \in \N$,
 points $\bX_\epsilon = \{\bx_1, \dots, \bx_{M}\}\subset\overline{\Omega}$,
 and operators
 \begin{equation*}
   \cG_{\epsilon} :
   \begin{cases}
   \cKabm\to Y \\
   a\mapsto  \left(\bRealiz_{\rho}(\Phiappx_{\epsilon}) \left( a(\bX_\epsilon) \right)  \right)\cdot\bRealiz_{\brho}(\Phirec_{\epsilon}) 
   \end{cases}
 \end{equation*}
where $\brho\in \{\ReLU, \ReLU^2\}^{\depth(\Phirec_{\epsilon})-1}$,
 $\bRealiz_{\brho}(\Phirec_{\epsilon}) : \Omega \to \R^{N}$,
 and ${\bRealiz_{\ReLU}}(\Phiappx_\epsilon) : \R^{M}\to \R^{N}$,
 satisfying
 \begin{equation*}
   \| \cS - \cG_{\epsilon} \|_{L^\infty(\cKabm, Y)} \leq \epsilon.
 \end{equation*}
As $\epsilon\to 0$, the encoder and decoder sizes are
 ${M(\epsilon)}  = \cO(\epsilon^{-2/m})$, ${N(\epsilon)} = \cO(\epsilon^{-2/m})$, 
and 
\begin{align*}
   &\depth(\Phiappx_{\epsilon})  = \cO\left( \left| \log\epsilon \right|^2 \right),\qquad
   \size(\Phiappx_{\epsilon}) = \cO\left(   \epsilon^{-6/m} \right),
   \\
   &\depth(\Phirec_{\epsilon}) = \cO\left( 1\right),
   \qquad
   \size(\Phirec_{\epsilon}) =\cO\left(  \epsilon^{- 4/m}\right).
 \end{align*}
\end{theorem}
\begin{proof}
  The proof follows the same line of reasoning as the proofs of Theorem
  \ref{th:fullONanalytic} and \ref{th:fullONanalytic-ReLU2}, by using 
  Proposition \ref{prop:general-ON} with $\Psi_N  = \bRealiz(\Phirec_\epsilon)$.
  \end{proof}
\begin{remark}
  The constant (with respect to the accuracy) depth of the decoding network is
  due to the wider range of activation functions used in its realization. 
  It is to be expected that a corresponding result with neural operators based strictly on 
  $\ReLU$ activation would include additional logarithmic terms in the 
  bound on the depth of the network.
\end{remark}
%%%%%%%%%%%%%%%%%%%%%%%%%%%%%%%%%%%%%%%%%%%%%%%%%%%%%%%%%%%%%%%%%%%%%%%%%%%%%%
\subsection{Nonlinear, non-smooth data-to-operator maps}
\label{sec:ExtLip}
%%%%%%%%%%%%%%%%%%%%%%%%%%%%%%%%%%%%%%%%%%%%%%%%%%%%%%%%%%%%%%%%%%%%%%%%%%%%%%
The present results were based on the
\emph{assumption made in Section~\ref{sec:VarFrm} 
that the data-to-operator map 
$a\mapsto L(a) $ is in $ \cL(X;\cL(Y,Y'))$, i.e., it is linear}.
Via compositionality,
the proofs developed here extend beyond this setting.
As an example for extension to a solution operator
for maps $a\mapsto L$ which are not linear (and non-smooth)
we consider the operator
$\tL : L^\infty(\Omega) \to \cL(Y, Y')$ such that,
for a real number $a_{\min{}} > 0$,
\begin{equation}
  \label{eq:nonlinL}
  \tL(a) u\coloneqq -\nabla\cdot((a_{\min{}} + |a|)\nabla u)
\end{equation}
with associated solution operator
\begin{equation*}
  \tcS :
  \begin{cases}
    L^\infty(\Omega) \to H^1_0(\Omega)\\
    a\mapsto u^a
  \end{cases}
\end{equation*}
where $u^a$ is the solution to
\begin{equation*}
  \tL(a)u^a = f \quad \text{in }\Omega,
\end{equation*}
with, e.g., homogeneous Dirichlet boundary conditions on $\partial\Omega$ and
with the previously made hypotheses on $\Omega$ and $f$.
The solution operator $\tcS$ is the composition of the analytic 
solution operator $\cS$ 
with a Lipschitz map $a\mapsto a_{\min{}} + |a|$.
% hence $\tcS$ is not analytic.
Nonetheless, since the absolute value can be emulated
exactly in a NN with $\ReLU$ activation, 
some results of Section \ref{sec:algconv} extend to this setting.
\begin{corollary}
Let $d=2$ and let $\Omega$ be a polygon with a finite number of sides.
Let, for $R>0$
\begin{equation*}
  \tcK = \{ v\in W^{1, \infty }(\Omega) : \|v\|_{W^{1,\infty}(\Omega)}\leq R\}.
\end{equation*}
Then,
for all $\epsilon \in (0,1)$, 
 there exist dimensions $N(\epsilon), M(\epsilon) \in \N$,
 points $\bX_\epsilon = \{\bx_1, \dots, \bx_{M(\epsilon)}\}\subset\overline{\Omega}$,
 and operators
 \begin{equation*}
   \tcG_{\epsilon} :
   \begin{cases}
   \tcK\to Y \\
   a\mapsto  \left(\bRealiz_{\ReLU}(\Phiappx_{\epsilon}) \left( a(\bX_\epsilon) \right)  \right)\cdot\bRealiz_{\brho}(\Phirec_{\epsilon}) 
   \end{cases}
 \end{equation*}
 where $\brho\in \{\ReLU, \ReLU^2\}^{\depth(\Phirec_{\epsilon})-1}$,
 $\bRealiz_{\brho}(\Phirec_{\epsilon}) : \Omega \to \R^{N}$,
 and ${\bRealiz_{\ReLU}}(\Phiappx_\epsilon) : \R^{M}\to \R^{N}$,
 satisfying
 \begin{equation*}
   \| \tcS - \tcG_{\epsilon} \|_{L^\infty(\tcK, Y)} \leq \epsilon.
 \end{equation*}
As $\epsilon\to 0$, 
the encoder and decoder sizes are
${M_\epsilon}  = \cO(\epsilon^{-2})$, ${N_\epsilon} = \cO(\epsilon^{-2})$, 
and
\begin{align*}
   &\depth(\Phiappx_{\epsilon})  = \cO\left( \left| \log\epsilon \right|^2 \right),\qquad
   \size(\Phiappx_{\epsilon}) = \cO\left(   \epsilon^{-6} \right),
   \\
   &\depth(\Phirec_{\epsilon}) = \cO\left( 1\right),
   \qquad
   \size(\Phirec_{\epsilon}) =\cO\left(  \epsilon^{-4}\right).
 \end{align*}
\end{corollary}
\begin{proof}
  The statement follows from Theorem \ref{th:fullONfinite-reg} (with $m=1$),
  since there exists a NN $\Phi^{\mathrm{abs}}$ such
  that
  \begin{equation*}
    \bRealiz_{\ReLU}(\Phi^{\mathrm{abs}})(a(\bx_1), \dots,
    a(\bx_{M_\epsilon}))_i = a_{\min{}} + | a(\bx_i) |, \qquad \forall i\in \{1, \dots, M_\epsilon\},
  \end{equation*}
  where $\{\bx_i\}_{i=1}^M$ are the encoding points.
  The rest of the proof follows directly from Theorem \ref{th:fullONfinite-reg},
  remarking that $\{v = a_{\min{}} + |a|:a \in \tcK\}$ is contained in a ball in $W^{1, \infty}(\Omega)$
\end{proof}
A formal and more systematic treatment of this (anecdotal) extension is the object of future work.
%%%%%%%%%%%%%%%%%%%%%%%%%%%%%%%%%%%%%%%%%%55
\section{Conclusion and Discussion}
\label{sec:Concl}
%%%%%%%%%%%%%%%%%%%%%%%%%%%%%%%%%%%%%%%%%%%%
We proved expression rate bounds for a class of neural operators $\cG$ 
approximating the coefficient-to-solution maps $\cS$ 
for elliptic boundary value problems of linear, 
divergence-form operators.
Specifically, 
we considered homogeneous Dirichlet boundary conditions,
in a bounded, polytopal domain $\Omega \subset \R^d$ 
in space dimension $d=2$
(and $d=3$ in Theorem \ref{th:fullONanalytic}) 
and the architecture \eqref{eq:G=RAE}, 
i.e.
$\cG = \cR\circ \cA \circ \cE$, with 
encoder $\cE$, decoder $\cR$ and approximator network $\cA$.
Our results emphasized proving the \emph{existence of neural operators
with poly-logarithmic and algebraic bounds on the number of parameters} with
respect to the (worst-case) error, for, respectively, 
analytic data and for data with finite regularity.
This constitutes a step in the direction of the full, 
mathematical analysis of these techniques.
Such analysis would also involve an analysis of their generalization properties and 
of the training algorithms used in practical computations.
The comparison with other polynomial, neural, and iteration-based operator surrogates such as \cite{FSZ25_1134}, is the topic of future work.
\appendix
%%%%%%%%%%%%%%%%%%%%%%%%%%%%%%%%%%%%%%%%%%%5
%
\section{Approximation theory with piecewise polynomials}
\subsection{Approximation of analytic functions in $\overline{\Omega}$}
\label{app:NwdthAn}
Let $d\in\{2,3\}$. 
For any set $U\subset\R^d$ and a positive constant $B>0$, 
we denote in this section
  \begin{equation*}
    \cA(U; B) = \left\{ v\in C^\infty(U) : \| v \|_{W^{k, \infty}(U)}
    \leq 
    B^{k+1}k!,\;\forall k\in\N_0 \right\}.
  \end{equation*}
  We show that there exist $C, b>0$ such that, for all $N\in\N$,
  \begin{equation*}
   d_N(\cA(\overline{\Omega}; B), L^\infty(\Omega)) 
   \leq C\exp(-bN^{1/d}).
  \end{equation*}
We are in Setting~\ref{set:analytic}, and assume that
$\Omega \subset \R^d$ is a 
bounded, polytopal domain with straight sides ($d=2$) resp. plane faces ($d=3$).

In either case, let $\cT = \{T \}$ denote a regular, 
finite partition of $\Omega$ into open, nondegenerate $d$-simplices $T$ 
(triangles if $d=2$ and tetrahedra if $d=3$).
This is to say that for all $T\in \cT$ it holds that $|T|>0$ 
and 
that each pair of simplices $T,T'\in \cT$ 
have closure-intersection $\overline{T}\cap \overline{T'}$ which is 
either empty, or an entire $k$-simplex with $0\leq k \leq  d-1$, 
so that $\cT$ is a simplicial complex.

By assumption, for all $T\in \cT$ it holds that $f\in \cA(\overline{T}; B)$.
Partition $T\in \cT$ into $d+1$ convex subdomains $Q_{T,i}$, $i=0,1,...,d$, 
which are the convex hull
of 
a) the barycenter of $T$, 
b) one vertex $v_T$ of the $d+1$ vertices of $T$, 
and 
c) all barycenters of the $d$ 
boundary simplices $T_\partial\subset \partial T$ which abut $v_T$. 
Then $f$, restricted to any of the $Q_{T,i}$, belongs to
$\cA(\overline{Q_{T,i}}; B)$.
Further, 
each $\overline{Q_{T,i}}$ is the image of the unit cube $[-1,1]^d$
under a $d$-linear mapping $G_{T,i}$: 
$$
\overline{Q_{T,i}} = G_{T,i}([-1,1]^d) \;.
$$
The $d$-linear maps $G_{T,i}$ are real-analytic in $[-1,1]^d$, 
being (component-wise) $d$-linear.
Due to $|T|>0$, also $|Q_{T,i}|>0$, whence $G_{T,i}$ is bijective.
The maps $G_{T,i}$ admit therefore bi-holomorphic extensions
to some open neighborhood $\widetilde{G_{T,i}} \subset \C^d$ such that
$\overline{G_{T,i}} \subset \widetilde{G_{T,i}}$, with 
${\rm dist}(\partial \widetilde{G_{T,i}}, \partial G_{T,i})>0$.

The assumption that the data $f$ be real analytic in $\overline{G_{T,i}}$, 
implies that the composition
$\hat{f}_{T,i} := f|_{G_{T,i}} \circ G_{T,i}$ 
is real-analytic in $\hat{Q} := [-1,1]^d$, for all $T\in \cT$ and $i=0,1,...,d$.
Hence, also $\hat{f}_{T,i}$ admits a holomorphic extension to a 
bounded, open neighborhood $\tilde{Q}$ of $\hat{Q}$ in $\C^d$, 
i.e., 
it holds that $\hat{Q} \subset \tilde{Q} \subset \C^d$ with strict inclusions.

This implies, by a tensor product argument, 
that the $d$-variate tensor product of the univariate 
Gauss-Legendre-Lobatto (GLL) interpolation operator $\cI_p$ in $[-1,1]$ 
of polynomial degree $p\geq 1$, 
$\cI_p^{\otimes d}$ ($d$-fold algebraic tensor product),
admits, for every fixed $k\geq 0$ and for every $p\geq k$, 
the exponential error bounds 
$$
\| \hat{f}_{T,i} - \cI^{\otimes d} \hat{f}_{T,i} \|_{W^{k,\infty}(\hat{Q})} 
\leq 
C_{T,i} \exp(-b_{T,i} p), \quad
p\in \N, \;T\in \cT, \;i\in \{0,1,\dots,d \}
$$
with positive constants $b_{T,i}, C_{T,i} > 0$ 
which in general depend on $k$, $T$, and $i$.

We transport the local GLL interpolants 
$\cI^{\otimes d} \hat{f}_{T,i} \in \mathbb{Q}_p^{d}$ to $Q_{T,i}$:
$$
f^p_{T,i} 
:= 
\left( \cI^{\otimes d} \hat{f}_{T,i} \right) \circ G_{T,i}^{(-1)}\;,
\quad 
T\in \cT, \; i=0,1,...,d \;,
$$
and obtain a Lipschitz-continuous, piecewise polynomial interpolant 
$f^p_T$ of total polynomial degree $pd$  on every $T\in \cT$
by assembling the $d+1$ local interpolants $f^p_{T,i}$ on the $Q_{T,i}$
(to verify continuity, we use the regularity of simplicial partition $\cT$, 
 and 
 that the one-sided traces of the interpolants on the boundaries $\partial Q_{T,i}$ 
 coincide and equal 
 the (affinely-transported) GLL interpolants of the boundary traces of $f_{T,i}$).

By the same argument, we further assemble the local GLL interpolants on each
$T\in \cT$ into a global in $\overline{\Omega}$ continuous, 
piecewise polynomial of (separate) degree $p\geq 1$
interpolant $f^p_{\cT}$ in $\Omega$.

We verify the exponential consistency bound.
For $k=0,1$, 
there are constants $b_{\cT}, C_{\cT} > 0$ which depend on $f$ and on $\Omega$, 
such that for all $p\geq 1$
\begin{align*}
\| f - f^p_{\cT} \|_{W^{k,\infty}(\Omega)} 
&= 
  \max_{T\in \cT} \| f_T - f^p_{T} \|_{ W^{k,\infty}(T)}\\
&\leq 
\max_{T\in \cT, i=0,1,...,d} \| f_T - f^p_{Q_{T,i}} \|_{ W^{k,\infty}(Q_{T,i}) }
\\
&\leq
\max_{T,i} C_{T,i} 
\exp(-(\min_{T,i} b_{T,i}) p) 
  \eqqcolon
C_{\cT} \exp(-b_{\cT} p) \;.
\end{align*}
% \qed
%
This completes the proof.
\subsection{Approximation rates for the set of solutions}
\label{sec:Nwidth-polytope}
Solutions of elliptic boundary value problems 
in polytopal domains $\Omega$ are known to belong to weighted Sobolev spaces,
with weights accounting for non-smoothness in the vicinity of
corners and edges of the boundary of the domain $\Omega$. 
We prove in Lemma \ref{lemma:nwidth-exp} 
the exponential decay of the approximation error for solution sets obtained with
analytic data. 
In Lemma \ref{lemma:FinReg2d} we address 
instead the case of data with finite regularity, 
and show algebraic rate bounds.

\begin{lemma}
  \label{lemma:nwidth-exp}
   Let $0<\beta<\alpha<\infty$, $A>0$, and $\kappa=3$ if $d=2$, $\kappa=5$ if $d=3$.
   Let $\Omega$ be a bounded, open polytope with Lipschitz boundary, 
   assume Setting \ref{setting:base-polytope}, and that the source term $f$ is analytic in $\overline{\Omega}$.
   Let 
  \begin{equation}
    \label{eq:Uanalytic-Nwidth}
    \cU = \cS\left(\left\{ v\in\cDab : \| v\|_{W^{m,\infty}(\Omega)} \leq A^{m+1}m!, \, \forall m\in\N_0\right\}  \right).
  \end{equation}
  Then, there exist $C, b>0$ such that for all $N\in\N$ there exists a space
  $V_N\in Y$ of piecewise polynomial functions such that $\dim(V_N) =N$ and
  \begin{equation*}
    \dist(\cU, V_N) \leq C\exp(-b N^{1/\kappa}).
  \end{equation*}
  When $N=2$, the meshes associated to the spaces $V_N$ are shape regular.
\end{lemma}
\begin{proof}
  For input data $a\in \cKab$ and for a source term $f$ analytic in $\overline{\Omega}$,
  the functions in $\cU$ are known to be weighted analytic in $\Omega$, with 
  corner-weights in polygons $\Omega$ in dimension $d=2$ (e.g. \cite{BG2d})
  with corner-edge weights in polyhedra $\Omega$ in dimension $d=3$ 
  (see, e.g. \cite{Costabel2012,BG3d-I,BG3d-II}). 
  In particular, given 
  $a\in\cKab = \left\{ v\in\cDab : \| v\|_{W^{m,\infty}(\Omega)} \leq A^{m+1}m!, \, \forall m\in\N_0\right\} $,
  there exists $M \geq 1 $ such that for all $\nu\in  \N_0^d$  
  the \emph{weighted analytic} estimates
  \begin{equation}
    \label{eq:bound-psi-analytic}
    \| w_\nu \partial^\nu \cS(a) \|_{L^2(\Omega)} \leq M^{|\nu|+1}|\nu|!
  \end{equation}
  hold, for a weight functions $w_{\nu}$ 
  that approach zero polynomially with respect to the distance from the corners and (when $d=3$)
  the edges of the boundary of the domain.

  Furthermore, by Remark 1.6.5 of \cite{Costabel2010} and inspecting the
  proof of the above mentioned Theorem 4.4 and Theorem 6.8 of \cite{Costabel2012},
  it can be seen that the constant $M$ in \eqref{eq:bound-psi-analytic} only depends on
  $\alpha, \beta, f$, on the constant $A$ in \eqref{eq:Uanalytic-Nwidth}, 
  and on the domain $\Omega$. 
  It can therefore be chosen uniformly over $\cU$.

  Classical results in $hp$ approximation theory (see, e.g., \cite{FeiSchhpGev,SchBook,SS15,BGhp3d,Opschoor2024})
  imply the result.
\end{proof}
% %
In the following Lemma \ref{lemma:FinReg2d} and Remark \ref{rmk:EllReg2d} 
we will employ the corner-weighted, 
Hilbert spaces considered by Kondrat'ev 
\begin{equation*}
  \calK^m_\gamma(\Omega) \coloneqq 
\left\{ v: r^{|\nu|-\gamma}\partial^\nu v \in L^2(\Omega),\, \forall |\nu|\leq m  \right\}.
\end{equation*}
defined for integer $m$ and $\gamma\in\R$. 
Here, 
for $x\in \Omega$, $r(x)$ denotes the distance to a corner of 
the polygon $\Omega$ situated nearest to $x$, 
and for
$\nu = (\nu_1, \nu_2)\in \N^2_0$ a multi-index
we have written 
$\partial^\nu = \partial_{x_1}^{\nu_1}\partial_{x_2}^{\nu_2}$ 
and 
$|\nu| = \nu_1+\nu_2$.
\begin{lemma} \label{lemma:FinReg2d}
  Let $0<\beta<\alpha<\infty$, $d=2$, $R>0$, $m\in \N$, 
  and let $\Omega$ be a
  plane polygon with boundary consisting of a finite number of straight sides.
  Assume Setting \ref{setting:base-polytope}, 
  that the right-hand side $f$ in \eqref{eq:sol-polytope} 
  has regularity 
  $f\in \calK^{m-1}_{\gamma-1}(\Omega)$ for some $\gamma > 0$ 
  and write
  \begin{equation*}
    \cU = \cS\left(\{v\in \cDab : \| v\|_{W^{m, \infty}(\Omega)} \leq R\}\right).
  \end{equation*}
  Then, there exists $C>0$ (depending on $\alpha,\beta,\gamma,\Omega$, $m$, $R$) 
  such that, for all $N\in \N$, there exist spaces $V_N\subset Y$ of piecewise
  polynomials functions such that $\dim (V_N) = N$ and
  \begin{equation*}
    \dist(\cU , V_N) \leq  C N^{-m/2}.
  \end{equation*}
\end{lemma}
\begin{proof}
Take the coefficient $a$ in \eqref{eq:problem} in the set
$\cDab\cap W^{m,\infty}(\Omega)$.
With the regularity $f\in \calK^{m-1}_{\gamma - 1}(\Omega)$ assumed here
for some $\gamma>0$ (see \cite{BNL17} and below),
we infer from \cite[Theorem 1.1]{BNL17}
that the weak solution $u^a\in Y$
corresponding to this data $a$ belongs to the corner-weighted
Kondrat'ev space $\calK^{m+1}_{\gamma+1}(\Omega)$,
with corresponding \textit{a priori} estimates.
In particular, 
there exists a constant $c>0$ that depends on $m$ such that
$$
\sup_{ a\in \cDab\cap W^{m, \infty}(\Omega)}
       \|u^a \|_{\calK^{m+1}_{\gamma+1}(\Omega)}
\leq c \| a \|_{W^{m,\infty}(\Omega)}^{n}
$$
% and on \cite{AdlerNistor15}.
for some integer exponent $n = n(m)$ specified in \cite[Theorem 1.1]{BNL17}.

If the admissible input data $a$ in \eqref{eq:problem} belong to some ball
in $W^{m, \infty}(\Omega)$,
the solutions $u^a = \cS(a)$ belong to a corresponding bounded subset
of $\calK^{m+1}_{\theta+1}(\Omega)$ for some $\theta>0$
(depending on the corner angles of $\Omega$ and on $\gamma$).

Functions in $Y\cap \calK^{m+1}_{\theta+1}(\Omega)$ are known to admit
optimal approximation rates from 
standard continuous, piecewise polynomial
Lagrangian finite element spaces of degree $m\geq 1$ 
on regular triangulations $\cT$ of $\Omega$ with
suitable mesh-refinement towards the corners of $\Omega$. 
Denoting the corresponding 
finite-dimensional subspaces of $Y=H^1_0(\Omega)$ by $\{ V^m_h \}_{h>0}$, 
with the parameter 
$h = \max\{ {\rm diam}(T): T\in \cT \}$ signifying the maximal diameter 
of triangles $T\in \cT$, 
in \cite[Theorem 4.4, Equation (19)]{BNZ2d} it is shown that 
there exists a constant $C(m,\theta,\Omega, \cKabm)>0$ such that 
$$
\inf_{u_h\in V^m_h} \| u-u_h\|_{H^1_0(\Omega)} 
\leq 
C(m,\theta,\Omega, \cD) ({\rm dim}(V^m_h))^{-m/2} \| u \|_{\calK^{m+1}_{\theta+1}(\Omega)}
\;.
$$
It follows that for the solution set
$\cU = \cS(\cDab\cap \{\|v\|_{ W^{m,\infty}(\Omega)}\leq R\})$
for some  $R>0$
it holds that
\begin{equation}
  \label{eq:FEapprox-m}
\sup_{a \in \cDab \cap \{\|v\|_{ W^{m,\infty}(\Omega)}\leq R\} }
\inf_{v_h \in V^m_h} 
\| u^a - v_h \|_{H^1_0(\Omega)}
\leq 
C 
(\dim(V^m_h))^{-m/2}
\end{equation}
where $C$ depends on $m,R,f,\Omega$.
\end{proof}
We add some remarks on the previous result.
\begin{remark}\label{rmk:EllReg2d}
The regularity theory developed in \cite{BNL17} allows for 
general, linear elliptic second order divergence form differential operators, 
i.e. 
\eqref{eq:problem} could have also first order differential and reaction terms.
It also admits so-called ``curvilinear polygonal'' $\Omega$, 
with possibly curved sides, and a finite number of (non-cuspidal) corners.

Homogeneous Dirichlet BCs in \eqref{eq:problem}
appear in the statement of \cite[Theorem 1.1]{BNL17}. 
However, \cite[Theorem 4.4]{BNL17} proves 
the same $\calK^{m+1}_{a+1}(\Omega)$ regularity 
also in the general setting of mixed homogeneous Dirichlet 
and conormal Neumann boundary conditions, 
as indicated in Remark~\ref{rmk:MixBCs}.

In the statement of Lemma~\ref{lemma:FinReg2d} we assumed finite
$W^{m,\infty}(\Omega)$ regularity of the input data, i.e., the 
diffusion coefficient $a$ in \eqref{eq:problem}.
The \emph{local regularity} $W_{\mathrm{loc}}^{m,\infty}(\Omega)$ 
is well-known to be essentially necessary for variational solutions
$u$ of \eqref{eq:problem} to belong to $H^{m+1}_{\mathrm{loc}}(\Omega)$, 
see, e.g. \cite{GilbargTrudinger}. 
However,
\cite[Theorems 1.1 and 4.4]{BNL17} establish the 
$$\calK^{m+1}_{a+1}(\Omega)$$ solution regularity for 
input data $a$ belonging to 
corner-weighted space, of integer order $m\geq 0$ 
defined by 
$$
\calW^{m,\infty}(\Omega) 
= 
\{ a:\Omega \to \R : r_{\Omega}^{|\nu|}\partial^\nu a \in L^\infty(\Omega), \;
   |\nu| \leq m\}
\;,\;\; m\in\N_0\;.
$$
These spaces are strictly larger than 
$W^{m,\infty}(\Omega)$ 
for every integer differentiation order $m > 0$.
The presently obtained NN emulation rates for \eqref{eq:problem} 
remain valid for inputs $a\in \calW^{m,\infty}(\Omega)$.

We addressed only $\Omega\subset \R^2$. 
Similar results hold in dimension $d=1$
and can be expected also in dimension $d=3$ 
(where proofs of $\calK^{m+1}_{a+1}(\Omega)$ solution regularity 
in polyhedral domains $\Omega$ only seem to be available for 
certain constant coefficient differential operators).
In these cases,
one could expect the asymptotic Kolmogorov $N$-width bound $O(N^{-m/d})$.
\end{remark}

\bibliographystyle{abbrv}
\bibliography{library}

\end{document}